%% file: Mechanism_v13_arXiv.tex
\documentclass[journal,twoside]{IEEEtran}  
\IEEEoverridecommandlockouts

\usepackage{epsf}
\usepackage{epsfig}
\usepackage{times}
\usepackage{amssymb}
\usepackage{amsmath}
\usepackage{amsfonts}
\usepackage{latexsym}
\usepackage{url}
\usepackage{mathrsfs}
\usepackage{algorithm}
\usepackage{bm}
\usepackage{algorithm,algorithmic}
\usepackage{subfigure}
\usepackage{framed} 
\usepackage[framed]{ntheorem}
\usepackage{cite}
\usepackage{array}      

\usepackage[usenames, dvipsnames]{color}
                  
\newframedtheorem{frm-definition}{Definition}

\newtheorem{theorem}{Theorem}

\newtheorem{lemma}{Lemma}

\newtheorem{remark}{Remark}
\newtheorem{assumption}{Assumption}

\newcommand{\hL}{\mathcal{L}}
\newcommand{\hN}{\mathcal{N}}

\newcommand{\cN}{{\cal N}}

\newcommand{\cE}{{\cal E}}

\newcommand{\cZ}{{\cal Z}}

\begin{document}
\title{An Incentive-Based Online Optimization Framework for Distribution Grids}
\author{Xinyang Zhou, Emiliano Dall'Anese, Lijun Chen, and  Andrea Simonetto 
\thanks{X. Zhou and L. Chen are with the College of Engineering and Applied Science, University of Colorado, Boulder, CO 80309, USA  (Emails: \{xinyang.zhou, lijun.chen\}@colorado.edu). E. Dall'Anese is with the National Renewable Energy Laboratory, Golden, CO 80401, USA (Email: emiliano.dallanese@nrel.gov). A. Simonetto is with IBM Ireland, Dublin (Email: andrea.simonetto@ibm.com).}
\thanks{This work was supported by the U.S. Department of Energy under Contract No. DE-AC36-08GO28308 with the National Renewable Energy Laboratory. The U.S. Government retains and the publisher, by accepting the article for publication, acknowledges that the U.S. Government retains a nonexclusive, paid-up, irrevocable, worldwide license to publish or reproduce the published form of this work, or allow others to do so, for U.S. Government purposes. }
\thanks{Preliminary result of this paper was presented at American Control Conference (ACC), Seattle, WA, May 2017 \cite{zhou2017pricing}.}
}

\maketitle

\begin{abstract}
This paper formulates a time-varying social-welfare maximization problem for distribution grids with distributed energy resources (DERs) and develops online distributed algorithms to identify (and track) its solutions. In the considered setting, network operator and DER-owners pursue given operational and economic objectives, while concurrently ensuring that voltages are within prescribed limits. The proposed algorithm affords an \emph{online} implementation to enable tracking of the solutions in the presence of time-varying operational conditions and changing optimization objectives. It involves a strategy where the network operator collects voltage measurements throughout the feeder to build incentive signals for the DER-owners in real time;  DERs then adjust  the generated/consumed powers in order to avoid the violation of the voltage constraints while maximizing given objectives. The stability of the proposed schemes is analytically established and numerically corroborated. 
\end{abstract}


\section{Introduction}

\input{intro_v13}

 \begin{table}
\begin{center}
\caption{Notation.}
\begin{tabular}{ll}   
\hline 
        $\cN$ & Set of nodes, excluding node $0$; $\cN:=\{1, ..., N\}$\\
        $\cE$ & Set of distribution lines\\
        $p_i$ & Net real power injected at node $i$\\
        $q_i$ & Net reactive power injected at node $i$\\
        $z_i$ & Overall power injected at node $i$, $z_i := [p_i,q_i]^{\intercal}$\\
        $\mathcal{Z}_i$ & Feasible set of real and reactive power at node i\\ 
        $p$ & $p := [p_1, \ldots, p_N]^{\intercal}$ \\
        $q$ & $q := [q_1, \ldots, q_N]^{\intercal}$ \\
        $z$ & $z := [p^{\intercal},q^{\intercal}]^{\intercal}$\\
        $V_i$ & Complex voltage at node $i$\\
        $v_i$  &  Voltage magnitude at node $i$\\
        $v$ & $v := [v_1, \ldots, v_N]^{\intercal}$ \\
        $\alpha_i$ & Signal for injected real power for node $i$\\
        $\beta_i$ & Signal for injected reactive power for node $i$\\
        $s_i$ & Overall signal $z_i := [\alpha_i,\beta_i]^{\intercal}$\\       
        $\alpha$ & $\alpha := [\alpha_1, \ldots, \alpha_N]^{\intercal}$ \\
        $\beta$ & $\alpha := [\beta_1, \ldots, \beta_N]^{\intercal}$ \\
        $s$ & Compact signal vector $s := [\alpha^{\intercal},\beta^{\intercal}]^{\intercal}$\\
        $[x]_+$ & Projection of $x$ onto the nonnegative orthant\\   
        $[x]_{\mathcal{Z}}$ & Projection of $x$ onto the convex set $\mathcal{Z}$ \\
        \hline 
\end{tabular}
\vspace{-.5cm}
\end{center}
\end{table}

\section{Preliminaries and System Model}
\label{sec:model}

\subsection{Network Model}

Consider a distribution network with $N+1$ nodes collected in the set $\cN \cup \{0\}$ with $\cN:=\{1, ..., N\}$ and node $0$ being the point of common coupling or substation, and distribution lines collected in the set $\cE$.
Let  $V^t_i \in \mathbb{C}$ denote the line-to-ground voltage at node $i$ at time $t$, and define $v^t_i := |V^t_i|$. Denote as $p^t_i \in \mathbb{R}$ and $q^t_i \in \mathbb{R}$  the (net) active and reactive power injections, respectively, of a distributed energy resource (DER) located at node $i \in \cN$.
For notational simplicity, exposition is tailored to the case where one DER is located at each node; however, the technical approach straightforwardly applies to the case where multiple DERs are connected to a node. {Hereafter, $\mathcal{Z}^t_i$ denotes the feasible set of active and reactive powers $p^t_i$ and $q^t_i$ at node $i \in \cN$ at time $t$. In the following, we explain how to construct this set for some types of DERs. }

\noindent \emph{Photovoltaic (PV) systems}: Let $p_{i,\text{av}}^{t}$ denote the available real power from a PV system at  time $t$, and let $\eta_i$ be the rated apparent capacity. Then, the set $\cZ_i^t$ is given by: 
\begin{align} 
\cZ_i^t =  \left\{(p_i, q_i) \hspace{-.1cm} :  0 \leq p_i  \leq  p_{i,\text{av}}^{t}, p_i^2 + q_i^2 \leq  \eta_i^2 \right\} .\nonumber
\end{align}

{
\noindent \emph{Energy storage systems}: The set $\cZ_i^t$ for an energy storage system is given by: 
\begin{align} 
\cZ_i^t =  \left\{(p_i, q_i) \hspace{-.1cm} :  \underline{p}_i^{t} \leq p_i  \leq  \overline{p}_i^{t}, p_i^2 + q_i^2 \leq  \eta_i^2 \right\} , \nonumber
\end{align} 
for given limits $\underline{p}_i^{t}, \overline{p}_i^{t}$ and for a given inverter rating $\eta_i$. The limits $\underline{p}_i^{t}, \overline{p}_i^{t}$ are updated during the operation of the battery based on the state of charge. 
}

\noindent \emph{Variable frequency drives}: For devices such as water pumps and supply fans of commercial HVAC systems, the set $\cZ_i^t$ can be described as:  
\begin{align} 
\cZ_i^t =  \left\{(p_i, q_i) \hspace{-.1cm} :  \underline{p}_i^{t} \leq p_i  \leq  \overline{p}_i^{t}, q_i = 0 \right\} , \nonumber
\end{align} 
for given  limits $\underline{p}_i^{t}, \overline{p}_i^{t}$. These limits can be fixed or updated by local controllers at a regular time intervals, based on the state of e.g., thermal loads.

{The operating region of small-scale diesel generators can be modeled using constant box constraints.  For DERs with discrete levels of output powers (e.g., electric vehicle chargers with discrete charging levels), $\cZ_i$ represents the convex envelope of the possible operating points; see e.g.,~\cite{AndreayOnlineOpt}. Randomization techniques can then be utilized to recover a feasible setpoint. However, the development of control strategies for DERs with discrete levels of output powers is left as a future research activity. }

Voltages, currents, and powers $\{p^t_i, q^t_i\}$ are related by the well-known nonlinear AC power-flow equations; assuming, for illustrative purpose, a balanced tree network, these equations read:   
\begin{subequations}\label{eq:bfm}
	\begin{eqnarray}
	P_{ij}^t \hspace{-2mm}&=& \hspace{-2mm}  - p^t_j +\!\!\!\! \sum_{k: (j,k)\in \cE}\!\!\!\!\! P^t_{jk}+  r_{ij}  \ell^t_{ij}  \label{p_balance}, \\[-3pt]
	Q^t_{ij} \hspace{-2mm}&=& \hspace{-2mm}  - q^t_j +\!\!\!\!\! \sum_{k: (j,k)\in \cE}\!\!\!\!\! Q^t_{jk} + x_{ij} \ell^t_{ij} \label{q_balance},\\[-3pt]
	{v^{t}_j}^2 \hspace{-2mm}&=& \hspace{-2mm}  {v^{t}_i}^2 - 2 \left(r_{ij} P^t_{ij} + x_{ij} Q^t_{ij}\right) + \left(r_{ij}^2+x_{ij}^2\right) \ell^t_{ij} \label{v_drop},\\
	\ell^t_{ij}{v^{t}_i}^2 \hspace{-2mm}&=& \hspace{-2mm}   {P_{ij}^{t}}^2 + {Q_{ij}^{t}}^2  \label{currents},
	\end{eqnarray}
\end{subequations}
where $\ell^t_{ij}$ is the squared magnitude of the current on line $(i,j)$, $P^t_{ij}, Q^t_{ij}$ are real and reactive powers injected on line $(i,j)$, and $r_{ij} + \mathrm{j} x_{ij}$ is the impedance on line $(i,j)$.

To facilitate the design and analysis of computationally-tractable algorithms, the proposed approach will employ suitable linearization approaches for~\eqref{eq:bfm}. Particularly, the following approximate linear relationship between voltage magnitudes and injected powers is utilized:
\begin{align} 
\hat{v}^t & = R p^t + X q^t + a , \label{eq:approximate} 
\end{align}
{where the parameters $R,X\!\!\! \in \mathbb{R}_{++}^{N\!\times\! N}$ and $a\!\!\in\mathbb{R}^{N}$
can be obtained using one of the two following approaches: }
i)~regression-based methods, based on real-time measurements of $\{v_i^t\}$, $p^t$, and $q^t$, e.g., the recursive least-squares  method~\cite{Angelosante10} can be utilized to continuously update the model parameters; and, ii) suitable linearization methods for the AC power-flow equations; see e.g.,~\cite{Baran89,farivar2013equilibrium,Christakou13,swaroop2015linear,bolognani2015linear}. In the latter case, the model parameters $R$, $X$, and $a$ can be time-varying too, by using current operating points as linearization points for the AC power-flow equations.  Parameters $R$, $X$, and $a$ should  be re-computed every time that the system changes topology. 

The approximate model~\eqref{eq:approximate} is utilized to facilitate the design of computationally-affordable algorithms. Section~\ref{sec:realtime} will  show how to leverage appropriate measurements to cope with approximation errors and systematically enforce voltage limits. 

{
\begin{remark} (Multiphase systems)
For notational and exposition simplicity, the framework is outlined for a single-phase system. However, the proposed algorithmic solution is applicable to unbalanced multiphase networks. This can be obtained by substituting~\eqref{eq:approximate} with the linearized model recently proposed in~\cite{bernstein2017load} for unbalanced multiphase networks  with both wye-connected  and delta-connected DERs.\hfill $\Box$
\end{remark}}

\subsection{Problem Setup}

The goal is to design a strategy wherein the network operator and end-customers pursue their own operational and economic objectives, while achieving a global coordination to enforce voltage regulation. 

\vspace{.1cm}

\subsubsection{End-customer optimization problem} 
Consider a cost function $C^t_i(p^t_i, q^t_i)$ that captures a well-defined performance objective for the customer(s) located at node $i \in \cN$ at time $t$. Let  $\alpha^t_i\in\mathbb{R}$ and $\beta^t_i\in\mathbb{R}$ be incentive signals produced  by the network operator (e.g., distribution system operator or aggregator) for  active and reactive power injections, respectively, at time $t$. Given signals $(\alpha_i^t, \beta^t_i)$, the following optimization problem is solved at each node $i\in\cN$ at time $t$: 
\begin{subequations}
\begin{eqnarray}
(\bm{\mathcal{P}^t_{1,i}}) &\underset{p^t_i, q^t_i}{\min}&f^t_i(p^t_i, q^t_i|\alpha^t_i, \beta^t_i),\label{eq:obj_c}\\
&\mathrm{s.t.}& (p^t_i, q^t_i) \in \mathcal{Z}^t_i,\label{eq:pccon_c}
\end{eqnarray}
\end{subequations}
where 
\begin{align}
f^t_i(p^t_i, q^t_i|\alpha^t_i, \beta^t_i):=C^t_i(p^t_i, q^t_i)-\alpha^t_i p^t_i-\beta^t_i q^t_i
\end{align}
with $\alpha^t_i p^t_i$ and $\beta^t_i q^t_i$ representing payment to/reward from the network operator. The following standard assumption is made. 

\vspace{.1cm}

\begin{assumption}
\label{as:Cinv}
Functions $C^t_i(p^t_i, q^t_i),\ \forall i\in\hN$ are continuously differentiable and strongly convex in $(p^t_i, q^t_i)$. Moreover, the first-order derivative of $C^t_i(p^t_i, q^t_i)$ is bounded in $\mathcal{Z}_i$.
\end{assumption}

\vspace{.1cm}

The assumption of bounded derivative means that an infinitesimal change in power should not lead to a jump in cost. 
Because (\ref{eq:obj_c}) is strictly convex in  $(p_i^t, q_i^t)$ and  $\mathcal{Z}^t_i$ is convex and compact,  a \emph{unique} solution  $(p_i^{t*}, q_i^{t*})$ exists for each $t$.

For future developments, consider the so-called best response strategy of node $i$, denoted as $b^t_i(\alpha^t_i,\beta^t_i)$, for given $\alpha^t_i$ and $\beta^t_i$:
\begin{eqnarray}
(p^{t*}_i,q^{t*}_i)=b^t_i(\alpha^t_i,\beta^t_i):=\underset{(p^t_i,q^t_i)\in\mathcal{Z}^t_i}{\arg \min}\ \ f^t_i\big(p^t_i,q_i^t | \alpha_i^t,\beta_i^t \big).\label{eq:br}
\end{eqnarray}

\vspace{.1cm}

\subsubsection{Social-welfare problem} 
Consider a cost function $D^t(\hat{v}^t)$ that captures network-oriented objective in voltage at time $t$. 
For example, to minimize the voltage deviation from the nominal value $v^{\text{nom}}$, we can set $D^t(\hat{v}^t)=\frac{1}{2}\|\hat{v}^t-v^{\text{nom}}\|^2$. The following assumption is made. 

\vspace{.1cm}

\begin{assumption}
\label{as:Dinv}
Function $D^t(\hat{v}^t)$ is continuously differentiable,  convex, and with bounded first-order derivative at achievable voltage magnitude values. 
\end{assumption}
Because the set of power injections $(p, q)$ is compact and $\hat{v}$ is a continuous function of $(p, q)$, the achievable $\hat{v}$ values are bounded. Thus, the boundedness of the first-order derivative of $D^t(\hat{v}^t)$ is a reasonable assumption. 

\vspace{.1cm}

Consider the following optimization problem, which captures both customer-oriented and network-oriented objectives of a distribution network: 
\begin{subequations}
\begin{eqnarray}
(\bm{\mathcal{P}^t_2}) &\underset{p^t,q^t,\hat{v}^t, \alpha^t,\beta^t}{\min}& \sum_{i\in\cN} C^t_i(p_i^t,q_i^t)+\gamma^t D^t(\hat{v}^t),\label{eq:obj}\\
&\mathrm{s.t.} & \hat{v}^t= Rp^t + Xq^t + a, \label{eq:volt0}\\
&&\underline{v}^t\leq \hat{v}^t \leq \overline{v}^t,\label{eq:volt}\\
&& (p_i^t,q_i^t)=b_i^t(\alpha_i^t,\beta_i^t),\ \forall i\in\cN,\label{eq:pccon}
\end{eqnarray}
\end{subequations}
where $\gamma^t \in \mathbb{R}_+$ is used to trade off between the end-customer and network-oriented objectives, and $\underline{v}^t$ and $\overline{v}^t$ are vectors collecting prescribed minimum and maximum voltage magnitude limits (the inequalities are component-wise). 

{Note that $(\bm{\mathcal{P}^t_2})$ is usually non-convex due to the constraint (\ref{eq:pccon}). This is because (\ref{eq:pccon}) is usually not affine.  
For better illustration of the non-convexity of $(\bm{\mathcal{P}^t_2})$, consider the following simple example with real power only. Assume a quadratic cost function $C^t_i(p_i^t)={p_i^{t}}^{2}$ and a box feasible set $p^t_i\in[\underline{p}_i,\overline{p}_i]$ with upper and lower bounds for real power injections $\underline{p}_i$ and $\overline{p}_i$, and we end up with a non-convex piece-wise linear function $b^t_i$:
\begin{eqnarray}
p^t_i=b^t_i(\alpha^t_i)=\left\{\begin{array}{ll}\underline{p}_i,   &\text{if}~ \alpha^t_i/2<\underline{p}_i \\ \alpha^t_i/2, & \text{if}~ \underline{p}_i\leq \alpha^t_i/2\leq \overline{p}_i \\ \overline{p}_i,   &\text{if}~  \alpha^t_i/2 > \overline{p}_i
\end{array}\right.,
\end{eqnarray}
which will become more complex if we consider more complicated $C_i^t$ and $\mathcal{Z}_i^t$.}

Problem $(\bm{\mathcal{P}^t_2})$ defines optimal operational \emph{trajectory} $\{p^{t *} ,q^{t *}\}_{t \in \mathbb{R}_+}$ over time for the active and reactive powers of the distribution network. One way to identify (and track) the time-varying optimal points of $(\bm{\mathcal{P}^t_2})$ consists in discretizing the temporal domain as $t_m : = m h,~k \in \mathbb{N}$, where  $h$ is a given time interval, and solve $(\bm{\mathcal{P}^t_2})$ at each time $t_m$. 
Section~\ref{sec:static} will focus on the case where distributed algorithms can be utilized to solve $(\bm{\mathcal{P}^{t_m}_2})$ \emph{to convergence} at each time $t_m$. These algorithms are suitable for operational conditions  where non-controllable demand/generation, ambient conditions, and cost functions are slow time-varying (i.e., the interval $h$ is ``large enough'' to allow convergence of the distributed algorithm).  Section~\ref{sec:realtime} will then focus on faster time-varying operational settings, and will advocate the development of \emph{online} algorithms~\cite{SimonettoGlobalsip2014} that track $\{p^{t *} ,q^{t *}\}_{t \in \mathbb{R}_+}$ over time.

\section{Incentive-Based Distributed Algorithm}\label{sec:static} 

Focusing on a particular problem instance at time $t$,  $(\bm{\mathcal{P}^t_2})$ lends itself to a Stackelberg game interpretation 
where $\alpha^t$ and $\beta^t$ are calculated via  $(\bm{\mathcal{P}^t_2})$  by the network operator (i.e., the leader) and broadcasted to all nodes $i\in\cN$;  subsequently, each end-consumer (i.e., the follower) computes the power setpoints $p_i^{t*}$ and $q_i^{t*}$ from $(\bm{\mathcal{P}^t_{1,i}})$. By design, $(p^{t*},q^{t*})$ is an optimal point of $(\bm{\mathcal{P}^t_2})$.  

However, it is challenging for the network operator to solve $(\bm{\mathcal{P}^t_2})$ not only because of the non-convexity introduced by constraint (\ref{eq:pccon}), but also because it requires knowledge  of the end-customer's best-response function $b_i^t$. To solve the problem, in Section~\ref{sec:reformulation} we first formulate a convex relaxation of $(\bm{\mathcal{P}^t_2})$ and show that its optimum gives the optimum of $(\bm{\mathcal{P}^t_2})$, and then in Section~\ref{sec:alg}  we design a distributed algorithm to solve  $(\bm{\mathcal{P}^t_2})$ based on the algorithm for the relaxed problem.  Since the same solution procedure is used to solve $(\bm{\mathcal{P}^t_2})$ \emph{to convergence} at each time $t$, the superscript $t$ will be dropped in this section. The  superscript $t$ will be re-introduced in Section~\ref{sec:realtime} where we outline an \emph{online} solution method.

\subsection{Convex Reformulation}
\label{sec:reformulation}
We start by deriving a convex relaxation of the non-convex problem $(\bm{\mathcal{P}_2})$ as well as conditions under which an optimal point of $(\bm{\mathcal{P}_2})$ can be identified. Consider the following convex optimization problem:
\begin{subequations}
\begin{eqnarray}
(\bm{\mathcal{P}_3}) &\underset{p,q,\hat{v}}{\min}& \sum_{i\in\cN} C_i(p_i,q_i)+\gamma D(\hat{v}),\label{eq:obj2}\\
&s.t.& \hat{v}= Rp + Xq + a,\label{eq:volt02}\\
&&\underline{v}\leq \hat{v} \leq \overline{v},\label{eq:volt2}\\
&& (p_i,q_i)\in \mathcal{Z}_i, \hspace{.5cm} \forall \, i \in \cN , \label{eq:pccon2}
\end{eqnarray}
\end{subequations}
where we replace the non-convex constraint (\ref{eq:pccon}) in $(\bm{\mathcal{P}_2})$ with (\ref{eq:pccon2}). We assume that the above problem is feasible. 

\begin{assumption}[Slater's condition]
\label{ass:sla}
There exists a feasible point $(\tilde{p},\tilde{q})\in \mathcal{Z}$, $ \mathcal{Z} :=  \mathcal{Z}_1 \times \ldots \times  \mathcal{Z}_N$, such that: 
\begin{align}
\underline{v} \leq R\tilde{p} + X\tilde{q} + a \leq \overline{v}.
\end{align}
\end{assumption}

Assumption~\ref{ass:sla} does not involve strict inequality because the constraint is linear. Given the strong convexity of the objective function \eqref{eq:obj2} in $(p_i, q_i)$ and the linear relation \eqref{eq:volt02}, a unique optimal solution exists for problem $(\bm{\mathcal{P}_3})$. Notice that a solution $(p_i^*, q_i^*, {\hat{v}}^{*})$ of $(\bm{\mathcal{P}_3})$ may not be feasible for $(\bm{\mathcal{P}_2})$, i.e., there does not exist a $(\alpha^*, \beta^*)$ such that $(p_i^*,q_i^*)=b_i (\alpha_i^*,\beta_i^*)$. We will, however, show next that such a $(\alpha^*, \beta^*)$ exists, and thus the solution of $(\bm{\mathcal{P}_3})$ gives the solution of $(\bm{\mathcal{P}_2})$.

Denote by $\underline{\mu}$ and $\overline{\mu}$ the dual variables  associated with the constraint~\eqref{eq:volt2}. Let $\hat{v}^*$ be the optimal voltage magnitudes produced by $(\bm{\mathcal{P}_3})$ and $\underline{\mu}^*,\overline{\mu}^*$ the optimal dual variables. Then, we propose to design the incentive signals as follows:
\begin{subequations}\label{eq:signal} 
	\begin{eqnarray}
		\alpha^* &=& R\big( \underline{\mu}^*-\overline{\mu}^*  -\gamma\nabla_{\hat{v}} D(\hat{v}^*)\big),\label{eq:alpha}\\
		\beta^* &=& X\big( \underline{\mu}^*-\overline{\mu}^* -\gamma\nabla_{\hat{v}} D(\hat{v}^*)\big) \, ,\label{eq:beta}
	\end{eqnarray}
\end{subequations}
where $\nabla_{\hat{v}} D$ denotes the gradient of function $D$ with respect to the vector $\hat{v}$. { 
Note that $\alpha^*$ and $\beta^*$ are composed of dual prices $\underline{\mu}^*, \overline{\mu}^*$ and the marginal cost of network operator $\gamma\nabla_{\hat{v}} D(\hat{v}^*)$, together with $R,X$ characterizing the network structure. As will be shown shortly, $\alpha^*$ and $\beta^*$ are de facto designed based on the optimality conditions of $(\bm{\mathcal{P}_2})$ and $(\bm{\mathcal{P}_3})$.} The above incentive signals are bounded, which precludes the possibility of infinitely large signals.

\begin{theorem}\label{lem:boundmu}
Under Assumptions \ref{as:Cinv}--\ref{ass:sla}, the incentive signals $(\alpha^*, \beta^*)$ defined by \eqref{eq:signal}  are bounded. \hfill $\Box$ 
\end{theorem}
\begin{proof}
Notice that the derivative $\nabla_{\hat{v}} D$ is bounded. To show the boundedness of $(\alpha^*, \beta^*)$, it is enough to show that the optimal duals $(\underline{\mu}^*, \overline{\mu}^*)$ are bounded. 

Consider the KKT conditions for problem $(\bm{\mathcal{P}_3})$:
	\begin{subequations}\label{eq:kktp3}
		\begin{eqnarray}
		&&\Big(\nabla_p\sum_{i\in\hN} C_i(p_i^*,q_i^*)+\gamma R\nabla_{\hat{v}} D(\hat{v}^*)-R(\underline{\mu}^*-\overline{\mu}^*)\Big)^{\intercal}\nonumber\\[-8pt]
		&& \hspace{3cm} (p-p^*)\geq 0, \forall (p,q)\in\mathcal{Z}, \label{eq:kkt1}\\
		&&\Big(\nabla_q\sum_{i\in\hN} C_i(p_i^*,q_i^*)+\gamma X\nabla_{\hat{v}} D(\hat{v}^*)-X(\underline{\mu}^*-\overline{\mu}^*)\Big)^{\intercal}\nonumber\\[-8pt]
		&& \hspace{3cm} (q-q^*)\geq 0,\ \forall (p,q)\in\mathcal{Z},\label{eq:kkt2}\\
		&&\hat{v}^*=Rp^*+Xq^*+a,\label{eq:kkt5}\\
		&&\underline{v}\leq \hat{v}^* \leq \overline{v}, \label{eq:kkt555}\\
		&&(\hat{v}^*-\underline{v})^{\intercal}\underline{\mu}^*=0,\ \underline{\mu}^*\geq 0,\label{eq:kkt55}\\
		&&(\overline{v}-\hat{v}^*)^{\intercal} \overline{\mu}^*=0,\ \overline{\mu}^*\geq 0.\label{eq:kkt6}
		\end{eqnarray}
	\end{subequations}
{ Combining (\ref{eq:kkt1})--(\ref{eq:kkt5}) results in:}
\begin{eqnarray}
&&\hspace{-4mm}\Big(\nabla_p\sum_{i\in\hN} C_i(p_i^*,q_i^*)+\gamma R\nabla_{\hat{v}} D(\hat{v}^*)\Big)^{\intercal}(p-p^*)\nonumber\\[-4pt]
&&\hspace{-1mm}+\Big(\nabla_q\sum_{i\in\hN} C_i(p_i^*,q_i^*)+\gamma X\nabla_{\hat{v}} D(\hat{v}^*)\Big)^{\intercal}(q-q^*)\nonumber\\[-2pt]
&&\hspace{-1mm}+(\overline{\mu}^*-\underline{\mu}^*)^{\intercal}(\hat{v}-\hat{v}^*)\geq 0,\ \forall (p,q)\in\mathcal{Z},\ \forall \hat{v},\label{kk3p3}
\end{eqnarray}
where the first two terms on the left of the inequality are bounded because of the bounded derivative of cost functions and the bounded set $\mathcal{Z}$. By the complementary slackness conditions \eqref{eq:kkt55}-\eqref{eq:kkt6}, $\overline{\mu}^*_i$ and $\underline{\mu}^*_i$,~$i\in \hN$ cannot be nonzero at the same time. If $\overline{\mu}^*_i\rightarrow\infty$, then $\hat{v}^*_i=\overline{v}_i$ and we can choose a $(p, q)$ and thus $\hat{v}_i$ such that the third term on the left of \eqref{kk3p3} goes to $-\infty$ and \eqref{kk3p3} does not hold. So, $\overline{\mu}^*_i$ and thus $\overline{\mu}^*$ is bounded. Similarly, we can show that $\underline{\mu}^*$ is bounded too. The result follows. 
\end{proof}

\vspace{.1cm}

By examining the optimality conditions of $(\bm{\mathcal{P}_2})$ and $(\bm{\mathcal{P}_3})$, we have the following result. 

\vspace{.1cm}
{
\begin{theorem}
\label{the1}
The solutions of problem $(\bm{\mathcal{P}_3})$ along with the signals $(\alpha^*, \beta^*)$ defined in (\ref{eq:signal}) are global optimal solutions of problem $(\bm{\mathcal{P}_2})$; i.e., problem  $(\bm{\mathcal{P}_3})$ is an exact convex relaxation of problem $(\bm{\mathcal{P}_2})$.  \hfill $\Box$
\end{theorem}
}
\begin{proof}
By the signal design (\ref{eq:signal}), (\ref{eq:kkt1})--(\ref{eq:kkt2}) become
	\begin{subequations}\label{eq:kktp1}
		\begin{align}
		&\hspace{-2mm}\Big(\nabla_p\sum_{i\in\hN} C_i(p_i^*,q_i^*)-\alpha^*\Big)^{\intercal}(p-p^*)\geq 0,\ \forall (p,q)\in\mathcal{Z},\label{eq:kkt3}\\[-0pt]
		&\hspace{-2mm}\Big(\nabla_q\sum_{i\in\hN} C_i(p_i^*,q_i^*)-\beta^*\Big)^{\intercal}(q-q^*)\geq 0,\ \forall (p,q)\in\mathcal{Z}.\label{eq:kkt4}
		\end{align}
	\end{subequations}
Notice that the above variational inequalities imply that $(p_i^*,q_i^*)=b_i (\alpha_i^*,\beta_i^*),~i\in\hN$. So, the solution of problem $(\bm{\mathcal{P}_3})$ along with $(\alpha^*, \beta^*)$ defined in (\ref{eq:signal}) is feasible for problem $(\bm{\mathcal{P}_2})$. The result follows, as $(\bm{\mathcal{P}_3})$ is a { convex} relaxation of $(\bm{\mathcal{P}_2})$. \end{proof}

\vspace{.1cm}

From now on, we will use the optima of $(\bm{\mathcal{P}_3})$  and $(\bm{\mathcal{P}_2})$ interchangeably depending on the context. Next, based on Theorem~\ref{the1}, we will develop an iterative algorithm that achieves the optimum of $(\bm{\mathcal{P}_3})$ (and hence that of $(\bm{\mathcal{P}_2})$) without exposing any private information of the end-customers to the network operator.

\begin{remark}
Theorem \ref{the1} asserts that non-convex problem $(\bm{\mathcal{P}_2})$ can be solved through solving a convex problem $(\bm{\mathcal{P}_3})$. 
At first glance, it appears that the non-convexity of $(\bm{\mathcal{P}_2})$ comes from a non-convex representation of the feasible set that may have a convex representation as implied by $(\bm{\mathcal{P}_3})$. An ongoing investigation is to identify the specific problem structure to generalize the result in Theorem \ref{the1} to a larger class  of problems. \hfill $\Box$
\end{remark}

\subsection{Distributed Algorithm}\label{sec:alg}

For notational simplicity, let $s_i= [\alpha_i, \beta_i]^{\intercal}$ denote the overall signals for end-customer $i$ and define $z_i= [p_i, q_i]^{\intercal}$. Further denote by $z:= [p^{\intercal}, q^{\intercal}]^{\intercal}\in\mathbb{R}^{2N}$ the vector of stacked power injections, and by $\mu:=[\underline{\mu}^{\intercal}, \overline{\mu}^{\intercal}]^{\intercal}\in\mathbb{R}_+^{2N}$ the vector of stacked dual variables.
Consider the following  Lagrangian function associated with $(\bm{\mathcal{P}_3})$:
\begin{eqnarray}
\hL(z,\mu)&=& \underset{i\in\cN}{\sum} C_i(z_i)+\gamma D(z)+\underline{\mu}^{\intercal}(\underline{v}-Rp-Xq-a) \nonumber\\[-4pt]
&&\hspace{7mm}+\overline{\mu}^{\intercal}(Rp+Xq+a-\overline{v}),  \label{eq:regLag}
\end{eqnarray}
which is obtained by keeping the constraints $z\in\mathcal{Z}$ and $\mu\in\mathbb{R}_+^{2N}$ implicit. Denote as $(z^*,\mu^*)$ a saddle-point of $\hL(z,\mu)$.  

To facilitate
the development of provably convergent online algorithms (the subject of Section~\ref{sec:realtime}), consider the following regularized Lagrangian function: 
\begin{eqnarray}
 \hL_{\phi}(z,\mu) &:= &  \sum_{i\in\cN} C_i(z_i)+\gamma D(z) +\underline{\mu}^{\intercal}(\underline{v}-Rp-Xq-a)\nonumber\\[-8pt]
&& \hspace{4mm}+\overline{\mu}^{\intercal}(Rp+Xq+a-\overline{v})-\frac{\phi}{2}\|\mu\|^2, \label{eq:regLag}
\end{eqnarray}
where $\phi>0$ is a predefined parameter (see e.g.,~\cite{Koshal11,SimonettoGlobalsip2014}). With the regularization term $-\frac{\phi}{2}\|\mu\|^2$,  the resultant function $\hL_{\phi}(z,\mu)$ is strongly concave in the dual variables. Based on~\eqref{eq:regLag}, {we proceed with the following minimax problem:}
\begin{eqnarray}
&\underset{\mu\in\mathbb{R}_+^{2N}}{\max}\ \underset{z\in\mathcal{Z}}{\min}& \hL_{\phi}(z,\mu) . \label{eq:lyareg}
\end{eqnarray}
In general, the unique optimizer of~\eqref{eq:lyareg}, denoted by $(z_{\phi}^*,\mu_{\phi}^*)$, is not  a saddle-point of the Lagrangian function~\eqref{eq:regLag} because of the regularization term $-\frac{\phi}{2}\|\mu\|^2$. However, the discrepancy between the unique optimizer of~\eqref{eq:lyareg} and the optimizers of~\eqref{eq:regLag} can be bounded as shown next.

Notice first that the boundedness of $\mu^*$ is shown in Theorem~\ref{lem:boundmu}; $\mu_{\phi}^*$ can be readily shown to be bounded too.
For ease of exposition, define $f(z):=\sum_{i\in\cN} C_i(z_i)+\gamma D(z)$ and $g(z):=\begin{bmatrix}\underline{v}-Rp-Xq-a\\ Rp+Xq+a-\overline{v}\end{bmatrix}$; this way,  the Lagrangian can be re-expressed in a compact form as $\hL(z,\mu)=f(z)+\mu^{\intercal}g(z)$ and the regularized counterpart reads $\hL_{\phi}(z,\mu)=f(z)+\mu^{\intercal}g(z)-\frac{\phi}{2}\|\mu\|^2$. From Assumption~\ref{as:Cinv}--\ref{as:Dinv}, it follows that $f$ is strongly convex in $z$. Equivalently, $\nabla_{z} f(z,\mu)$ is strongly monotone in $z$. {Therefore, we have the following lemma:
\begin{lemma}
There exists a scalar $c>0$ such that $\forall z,z'\in\mathcal{Z}$,
\begin{eqnarray}
\big(\nabla_{z} f(z,\mu)-\nabla_{z} f(z',\mu)\big)^{\intercal}(z-z')\geq c\|z-z'\|^2.\label{eq:strongmono}
\end{eqnarray}
	\hfill $\Box$ 
\end{lemma}
}

Then, the discrepancy between  $z^*$ and $z_{\phi}^*$ due to the regularization term can be bounded as follows (see also~\cite[Proposition~3.1]{Koshal11}).
\begin{theorem}\label{thm:btr}
	The difference between $z_{\phi}^*$ and $z^*$ is bounded as:
	\begin{eqnarray}
	\|z_{\phi}^*-z^*\|^2\leq \frac{\phi}{2{c}}\big(\|\mu^*\|^2-\|\mu_{\phi}^*\|^2\big).\label{eq:discrepancy}
	\end{eqnarray}
	\hfill $\Box$ 
\end{theorem}
\begin{proof}
	As a saddle point of (\ref{eq:lyareg}), $(z_{\phi}^*,\mu_{\phi}^*)$ satisfies the following inequalities: 
	\begin{eqnarray}
	\hL_{\phi}(z_{\phi}^*,\mu)\leq \hL_{\phi}(z_{\phi}^*,\mu_{\phi}^*)\leq \hL_{\phi}(z,\mu_{\phi}^*),~\forall z,\mu\label{eq:saddlepoint} \, .\nonumber
	\end{eqnarray}
	 The left inequality leads to
	\begin{eqnarray}
	(\mu_{\phi}^*-\mu^*)^{\intercal}g(z_{\phi}^*)-\frac{\phi}{2}\|\mu_{\phi}^*\|^2+\frac{\phi}{2}\|\mu^*\|^2\geq 0,\label{eq:dis0}
	\end{eqnarray}
	where we set $\mu=\mu^*$. We next characterize the term $(\mu_{\phi}^*-\mu^*)^{\intercal}g(z_{\phi}^*)$.
	
	\emph{(i)} Leveraging the definition of convex functions, $g_j(z_{\phi}^*)$ can be upper bounded as:
	\begin{eqnarray}
	g_j(z_{\phi}^*)&\leq& g_j(z^*)+\nabla_z g_j(z_{\phi}^*)^{\intercal}(z_{\phi}^*-z^*)\nonumber\\
	&\leq&\nabla_z g_j(z_{\phi}^*)^{\intercal}(z_{\phi}^*-z^*),\label{eq:dis1}
	\end{eqnarray}
	where the second inequality is due to the fact that $g_j(z^*)\leq 0$. Multiply both sides of (\ref{eq:dis1}) by $\mu_{\phi,j}^*$ (which is nonnegative) and sum up for all $j$ to have:
	\begin{eqnarray}
	\mu_{\phi}^{*\intercal}g(z_{\phi}^*)&\leq& \sum_j \mu_{\phi,j}^*\cdot\nabla_zg_j(z_{\phi}^*)^{\intercal}(z_{\phi}^*-z^*) \nonumber\\[-4pt]
	&=&\!\!\!\nabla_z\hL_{\phi}(z_{\phi}^*,\mu_{\phi}^*)^{\intercal}(z_{\phi}^*-z^*)\!-\!\nabla_z f(z_{\phi}^*)^{\intercal}(z_{\phi}^*-z^*)\nonumber\\
	&\leq&-\nabla_z f(z_{\phi}^*)^{\intercal}(z_{\phi}^*-z^*),\label{eq:dis2}
	\end{eqnarray}
	where the second inequality is due to the first-order optimality condition $\nabla_z\hL_{\phi}(z_{\phi}^*,\mu_{\phi}^*)^{\intercal}(z_{\phi}^*-z^*)\leq 0$.

	\emph{(ii)} On the other hand, one has that:
 \begin{eqnarray}
	g_j(z_{\phi}^*)\geq g_j(z^*)+\nabla_zg_j(z^*)^{\intercal}(z_{\phi}^*-z^*)\label{eq:dis3}.
	\end{eqnarray}
	Multiply both sides of (\ref{eq:dis3}) by $-\mu_{j}^*$ (which is nonpositive) and sum up for all $j$ to {get}:	
	\begin{eqnarray}
	-\mu^{*\intercal}g(z_{\phi}^*)&\leq& \!\!\!\!\!-\!\sum_j \mu_j^* g_j(z^*)\!-\!\sum_j\mu_j^*\cdot\nabla_zg_j(z^*)^{\intercal}(z_{\phi}^*-z^*)\nonumber\\[-3pt]
	&=&\sum_j\mu_j^*\cdot\nabla_zg_j(z^*)^{\intercal}(z^*-z_{\phi}^*)\nonumber\\[-3pt]
	&=&\!\!\!\!\nabla_z\hL(z^*,\mu^*)^{\intercal}(z^*-z_{\phi}^*)\!-\!\nabla_z f(z^*)^{\intercal}(z^*-z_{\phi}^*)\nonumber\\
	&\leq& \nabla_z f(z^*)^{\intercal}(z_{\phi}^*-z^*),\label{eq:dis4}
	\end{eqnarray}
	where the first equality is due to the complimentary slackness condition and the second inequality is obtained from the first-order optimality condition.
	
	Substitute (\ref{eq:dis2}) and (\ref{eq:dis4}) into (\ref{eq:dis0}), and use (\ref{eq:strongmono}) to obtain (\ref{eq:discrepancy}).
\end{proof}

The key advantage of utilizing the regularized Lagrangian is that the primal-dual gradient methods applied to~\eqref{eq:lyareg} exhibit improved convergence properties~\cite{SimonettoGlobalsip2014} as explained next. 

Hereafter, we omit the subscript $\phi$ from the optimization variables for notational simplicity, with the understanding that the updates of $z(k)$ and $\mu(k)$ are designed to solve the regularized saddle-point problem~\eqref{eq:lyareg}.  Consider the following primal-dual projected gradient method, where $k$ denotes the iteration index: 
\begin{eqnarray}
&&\hspace{-14mm}\begin{bmatrix}z(k+1)\\ \mu(k+1)\end{bmatrix} = \  \hat{T}\left(\begin{bmatrix}z(k)\\ \mu(k)\end{bmatrix}\right)\nonumber\\
&:=&\left[\begin{bmatrix}z(k)\\ \mu(k)\end{bmatrix} -\begin{bmatrix}\varepsilon_1\nabla_{z} \hL_{\phi}(z(k),\mu(k)) \\ -\varepsilon_2\nabla_{\mu}\hL_{\phi}(z(k),\mu(k))
\end{bmatrix} \right]_{\mathcal{Z}\times \mathbb{R}^{2N}_+}\!\!\!\!,\label{eq:mappingT}
\end{eqnarray}
where $[~]_{\mathcal{Z}\times \mathbb{R}^{2N}_+}$ denotes the projection operation onto the set ${\mathcal{Z}\times \mathbb{R}^{2N}_+}$, and $\varepsilon_1, \varepsilon_2 > 0$ are prescribed step sizes for the primal and the dual updates.
Notice that  $\nabla_{z} \hL_{\phi}(z,\mu)$ and $\nabla_{\mu}\hL_{\phi}(z,\mu)$ are  Lipschitz continuous and strongly monotone. Therefore by virtue of \cite[Sec.~3.5, Proposition~5.4]{BeT89}, the following result holds. 

\begin{theorem}\label{the:convergence}
There exist some $\bar{\varepsilon}_1,\bar{\varepsilon}_2>0$ such that for any $\varepsilon_1\in(0,\bar{\varepsilon}_1],\varepsilon_2\in(0,\bar{\varepsilon}_2]$, $\hat{T}$ is a contraction mapping. For $\varepsilon_1\in(0,\bar{\varepsilon}_1],\varepsilon_2\in(0,\bar{\varepsilon}_2]$, the sequence $\{(z(k),\mu(k))\}$ generated by  (\ref{eq:mappingT}) converges geometrically to the optimizer of (\ref{eq:lyareg}).\hfill $\Box$
\end{theorem}

{
The proof is referred to \cite{BeT89} and omitted here. We can further provide analytical bound for such $\bar{\varepsilon}_1,\bar{\varepsilon}_2$ for completeness. The results put as Theorem~\ref{the:stepsizebound} are presented in the Appendix for better readability. We also refer to Section~\ref{sec:simconv}) for some numerical characterization of step sizes as regards convergence.
}

Given Theorems~\ref{thm:btr}--\ref{the:convergence},  algorithm \eqref{eq:mappingT} converges to within a small neighborhood of problem $(\bm{\mathcal{P}_3})$ (problem $(\bm{\mathcal{P}_2})$) whose size can be controlled by choosing a proper weight $\phi$ for the regularization term. 

The decomposable structure of (\ref{eq:mappingT}) naturally enables the  following iterative \emph{distributed} algorithm:
\begin{subequations}\label{eq:dyn}
		\begin{eqnarray}
		\!\!\!\!\!\!\!\!z_i(k+1)\!\!\!\!&=&\!\!\!\! \big[z_i(k)-\varepsilon_1\big(\nabla_z C_i(z_i(k))-s_i(k)\big)\big]_{\mathcal{Z}_i},\label{eq:dyncus} \\
		\!\!\!\!\!\!\!\!\underline{\mu}(k+1)\!\!\!\!&=&\!\!\!\! \big[\underline{\mu}(k)+\varepsilon_2\big(\underline{v}-\hat{v}(k)-\phi\underline{\mu}(k)\big)\big]_+\,,\label{eq:dynmu}\\
		\!\!\!\!\!\!\!\!\overline{\mu}(k+1)\!\!\!\!&=&\!\!\!\! \big[\overline{\mu}(k)+\varepsilon_2\big(\hat{v}(k)-\overline{v}-\phi\overline{\mu}(k)\big)\big]_+\,,\\
		\!\!\!\!\!\!\!\!\alpha(k+1)\!\!\!\!&=&\!\!\!\!R\big[\underline{\mu}(k+1)\!-\!\overline{\mu}(k+1)\!-\!\gamma\nabla_{\hat{v}}D(\hat{v}(k))\big]\,,\label{eq:dynalpha}\\
		\!\!\!\!\!\!\!\!\beta(k+1)\!\!\!\!&=&\!\!\!\! X \big[\underline{\mu}(k+1)\!-\!\overline{\mu}(k+1)\!-\!\gamma\nabla_{\hat{v}}D(\hat{v}(k))\big]\,,\label{eq:dynbeta}\\
		\!\!\!\!\!\!\!\!\hat{v}(k+1)\!\!\!\!&=&\!\!\!\! Rp(k+1)+Xq(k+1) + a\,, \label{eq:dynv}
		\end{eqnarray}
\end{subequations}
where the power setpoints of each device are computed locally through~\eqref{eq:dyncus} and~\eqref{eq:dynmu}--\eqref{eq:dynv} are performed at the network operator. The resultant scheme is tabulated as Algorithm~1. Notice that each end-customer $i$ \emph{does not} share its cost function $C_i$ \emph{or} its feasible set $\mathcal{Z}_i$ with the network operator; rather, the end-customer transmits to the network operator only the resultant power injections $z_i(k)$. Indeed, the results of Theorem~\ref{thm:btr}--\ref{the:convergence} apply to (\ref{eq:dyn}) too. 

{
\begin{remark}
In (\ref{eq:dyn}),  $\alpha$ and $\beta$ are utilized by the end-customers to construct  the  primal gradient $\nabla_z\hL_{\phi}$. This strategy enables a distributed implementation of the primal-dual projected gradient method (\ref{eq:mappingT}) without the end-customers knowing information of the network.\hfill $\Box$
\end{remark}
}

\begin{algorithm}[t]
\label{alg:mechanism_2}
\caption{Incentive-based iterative algorithm} 
\begin{algorithmic}

\REPEAT 

\STATE [S1] End-customer $i \in \cN$ performs~\eqref{eq:dyncus} and sends $z_i(k+1)$ to network operator.

\STATE [S2] Network operator  performs steps~\eqref{eq:dynmu}--\eqref{eq:dynv}.

\STATE [S3] Network operator transmits signals $s_i(k+1)$ to end-customer $i \in \cN$.

\UNTIL stopping criterion is met

\end{algorithmic}
\end{algorithm}

\section{Online Algorithm}\label{sec:realtime}

Algorithm~1 involves an iterative procedure where \eqref{eq:dyn} is repeated until convergence to (within a small neighborhood of) 
an optimum of problem $(\bm{\mathcal{P}_2^{t}})$ 
at each time $t$. 
This requires considering a discrete-time model where time is divided into slots of equal duration, indexed by $t_{m} \in \mathbb{N}_+ = \{0, 1, 2 \cdots \}$, 
where the timeslot duration is expected to be much longer than the convergence time of Algorithm~1.  However, in the case of fast changing operational conditions and cost functions, it is desirable to use a small timeslot duration, i.e., to sample $(\bm{\mathcal{P}_2^{t}})$ at a small sampling interval to track $\{{p^t}^*, {q^t}^* \}_{t \in \mathbb{R}_+}$~\cite{SimonettoGlobalsip2014}. In this case, it may not be possible to solve $(\bm{\mathcal{P}_2^{t_{m}}})$ to convergence within the timeslot, and instead only $K > 0 $ iterations may be performed. Also, problem $(\bm{\mathcal{P}_2^{t_{m}}})$ uses linear approximation \eqref{eq:approximate} of the power-flow model. 
In this section, we will develop an online algorithm that continuously pursues the optima of $(\bm{\mathcal{P}_2^{t_{m}}})_{t_{m} \in \mathbb{N}_+}$, and characterize the ``loss'' of optimality caused by the finite iterations as well as the approximation error.  

\subsection{Online Algorithm}

As in Section \ref{sec:alg}, the voltage magnitudes $v_i,~i\in\hN$ will be measured, but they will follow the power-flow equations \eqref{eq:bfm} instead of its linear approximation \eqref{eq:approximate}. We assume that the approximation error is bounded (see also \cite{chiang1990existence, zhou2016vvac}). 

\begin{assumption}
\label{ass:error}
There exists a constant $e>0$ such that
$|v_i^{t_m} (z)-\hat{v}_i^{t_m} (z)|\leq e,\ i\in\hN$ for all $z\in\mathcal{Z}$ at any time $t_m$.
\end{assumption}

As the linearized power-flow model \eqref{eq:approximate} is a very accurate approximation under normal operating condition\cite{Baran89,swaroop2015linear,bolognani2015linear}, the bound $e$ expects to be small. 

With the measurement of the voltage magnitudes, the proposed algorithm, formally described as Algorithm \ref{alg:mechanism_3}, executes the following steps at time $t_{m}$ ($k$ denotes the iteration index): 
\begin{subequations}
\vspace{-.2cm}
\label{eq:online}
	\begin{eqnarray}
	{z}^{t_{m}}_i(k+1)  \hspace{-2mm} &=& \hspace{-2mm} \big[{z}^{t_{m}}_i(k)
	-\varepsilon^{t_{m}}_1\big(\nabla_z C_i^{t_{m}}({z}^{t_{m}}_i(k))\nonumber\\
	&& \hspace{-2mm}~~-{s}^{t_{m}}_i(k)\big)\big]_{\mathcal{Z}_i^{t_{m}}}~,\label{eq:dyncus3}\\
	{\underline{\mu}}^{t_{m}}(k+1) \hspace{-2mm} &=& \hspace{-2mm} \big[{\underline{\mu}}^{t_{m}}(k)+\varepsilon^{t_{m}}_2\big(\underline{v}^{t_{m}}-{v}^{t_{m}}(k)\nonumber\\
		&& \hspace{-2mm}~~ -\phi{\underline{\mu}}^{t_{m}}(k)\big)\big]_+~,\label{eq:dynmu3}\\
	{\overline{\mu}}^{t_{m}}(k+1) \hspace{-2mm}&=& \hspace{-2mm}\big[{\overline{\mu}}^{t_{m}}(k)+\varepsilon^{t_{m}}_2\big({v}^{t_{m}}(k)-\overline{v}^{t_{m}}\nonumber\\
			&& \hspace{-2mm}~~ -\phi{\overline{\mu}}^{t_{m}}(k)\big)\big]_+~,\\
	{\alpha}^{t_{m}}(k+1) \hspace{-2mm}&=& \hspace{-2mm}R\big[{\underline{\mu}}^{t_{m}}(k+1)-{\overline{\mu}}^{t_{m}}(k+1)\nonumber\\
	 \hspace{-2mm}&& \hspace{-2mm}~~-\gamma^{t_{m}}\nabla_vD^{t_{m}}({v}^{t_{m}}(k))\big]\,~,\\
	{\beta}^{t_{m}}(k+1) \hspace{-2mm}&=& \hspace{-2mm}X \big[{\underline{\mu}}^{t_{m}}(k+1)-{\overline{\mu}}^{t_{m}}(k+1)\nonumber\\
	 \hspace{-2mm}&& \hspace{-2mm}~~-\gamma^{t_{m}}\nabla_vD^{t_{m}}({v}^{t_{m}}(k))\big]\,~,\label{eq:dynbeta3}\\
	 {v}^{t_{m}}(k+1) \hspace{-2mm}&&\hspace{-2mm} \text{obey the nonlinear model (\ref{eq:bfm})}.
	\end{eqnarray}
\end{subequations}
Iterations~\eqref{eq:online} are performed $K > 0$ times during each timeslot $t_{m}$. When $K=1$, only one iteration~\eqref{eq:online} is computed per timeslot~\cite{SimonettoGlobalsip2014,OPFpursuit,LowOnlineOPF,AndreayOnlineOpt,Bernstein15}. { It is worth noticing that a centralized implementation of~\eqref{eq:online} requires collecting the time-varying $C_i^t$ and $\mathcal{Z}_i^t$ at the network operator at each iteration; hence, a a centralized implementation would incur a higher communication overhead.}

\begin{algorithm}
\caption{Real-time incentive-based algorithm } \label{alg:mechanism_3}
\begin{algorithmic}
\STATE At each timeslot $t_{m}$:
\STATE ~\,\,\,\,[S0] Initialization: $z^{t_{m}}(0)=z^{t_{m} -1}(K)$, \\ 
\hspace{1.1cm} $\mu^{t_{m}}(0)=\mu^{t_{m}-1}(K)$,  \\
\hspace{1.1cm} $s^{t_{m}}(0)=s^{t_m-1}(K)$.

\REPEAT 

\STATE [S1] End-customer $i \in \cN$ performs~\eqref{eq:dyncus3}.  

\STATE [S2] End-customer $i \in \cN$ {implements} $z^{t_{m}}_i(k+1)$. 

\STATE [S3] Network operator performs steps~\eqref{eq:dynmu3}--\eqref{eq:dynbeta3}.

\STATE [S4] Network operator transmits signals $s^{t_{m}}_i(k+1)$ to end-customer $i \in \cN$.

\STATE [S5] Network operator measures voltages $v^{t_{m}}$.

\UNTIL $k= K$

\end{algorithmic}
\end{algorithm}

In the next subsection, we will analyze the convergence and tracking capability  of the above online algorithm. 

{
\begin{remark} (local controller)
The proposed algorithms  produce setpoints $(p_i^{t}, q_i^{t}) \in \cZ_i^{t}$ for the output powers of the DERs. It is assumed that the DERs are endowed with local controllers that are designed so that, upon receiving the setpoint, the output powers are driven to the commanded setpoints. Relevant dynamical models for the output powers of inverters operating in a grid-connected mode are discussed in e.g.,~\cite{Iravanibook10,Irminger12} and can be found in  datasheets of commercially available DERs. Assumption~\ref{ass:error} accounts for both  measurement errors and bounds the discrepancy between the commanded setpoint and the actual output powers; Assumption~\ref{ass:error} is valid, for example, when the DER's response to a step-change in the setpoint follows a first-order model~\cite{Iravanibook10,Irminger12}. \hfill $\Box$
\end{remark}}

{
\begin{remark} (implementation)
The proposed real-time algorithms update the setpoints of the DERs on a second or subsecond timescale to  maximize the operational objectives while coping with the variability of available renewable-based generation and non-controllable energy assets. The algorithm does not control any anti-islanding and ride-through parameters. Considerations regarding the recloser-fuse coordination problem (which affects anti-islanding and ride-through configurations) pertain to  the deployment of the DERs and given interconnection agreements. In case of event where the DERs are required to shut off, the algorithm will not produce any setpoint; the algorithm will re-start producing setpoints once the DERs are allowed to reconnect to the system.  \hfill $\Box$
\end{remark}}

\subsection{Performance Analysis}

In this subsection, the hatted symbols (e.g.,  $\hat{\alpha}^t(k)$, $\hat{\beta}^t(k)$) refer to the iterates produced by the algorithm~\eqref{eq:dyn} (i.e., under the linear approximation \eqref{eq:approximate}), while the non-hatted symbols (e.g., $\alpha^t(k)$, $\beta^t(k)$) refer to those produced by~\eqref{eq:online} (i.e., under the nonlinear model (\ref{eq:bfm})).

By comparing~\eqref{eq:dyn} and~\eqref{eq:online}, Assumption~\ref{ass:error} leads to the following bounds: 
\begin{eqnarray}
|\hat{\underline{\mu}}^{t_m}_i-\underline{\mu}^{t_m}_i|\leq \varepsilon_2 e,\ \ \ |\hat{\overline{\mu}}_i^{t_m}-\overline{\mu}_i^{t_m}|\leq \varepsilon_2 e,\nonumber\\
|\hat{\alpha}_i^{t_m}-\alpha_i^{t_m}|\leq R_i^{\intercal}(\gamma\nabla_v^2D(\tilde{v}^{t_m})\bm{1}_n+\varepsilon_2)e,\nonumber\\
|\hat{\beta}_i^{t_m}-\beta_i^{t_m}|\leq X_i^{\intercal}(\gamma\nabla_v^2D(\tilde{v}^{t_m})\bm{1}_n+\varepsilon_2)e,\nonumber
\end{eqnarray}
for some $\tilde{v}^{t_m}$ and, therefore:
\begin{eqnarray}
|\hat{p}^{t_m}_i-p^{t_m}_i|\leq \varepsilon_1 R_i^{\intercal}(\gamma\nabla_v^2D(\tilde{v}^{t_m})\bm{1}_n+\varepsilon_2)e:=\delta_{1,i},\nonumber\\
|\hat{q}^{t_m}_i-q^{t_m}_i|\leq \varepsilon_1 X_i^{\intercal}(\gamma\nabla_v^2D(\tilde{v}^{t_m})\bm{1}_n+\varepsilon_2)e:=\delta_{2,i}.\nonumber
\end{eqnarray}

Let $\delta:=[\delta_{1,1},\ldots,\delta_{1,N},\delta_{2,1},\ldots,\delta_{2,N}]\in\mathbb{R}_+^{2N}$, and collect the primal and dual variables {in} the vector $y:=(z,\mu)$ for notational simplicity. Consequently, {the following holds}: 
\begin{eqnarray}
\|\hat{T}^{t_m}(y)-T^{t_m}(y)\|\leq \|\rho\|,\ \forall y\in\mathcal{Z}^{t_m}\times \mathbb{R}^{2N}_+,\label{eq:bound2model}
\end{eqnarray}
where $\rho:= [\varepsilon_2 e\cdot \mathbf{1}_{1\times 2N},\delta^{\intercal}]^{\intercal}$ and $T^{t_m}(\cdot)$ is the counterpart of $\hat{T}^{t_m}(\cdot)$ for the iterates~\eqref{eq:online} at time $t_m$. Let $\Delta\leq\bar{\Delta}<1$ be the contraction modulus for $\hat{T}^{t_m}(\cdot)$; with appropriate step sizes $\varepsilon^{t_m}_1$ and $\varepsilon^{t_m}_2$ chosen according to Theorem \ref{the:convergence}, by definition, we have that: 
\begin{eqnarray}
\|\hat{T}^{t_m}(y)-\hat{T}^{t_m}(y')\|\leq\Delta\|y-y'\|,\ \forall y,y'\in\mathcal{Z}^{t_m}\times \mathbb{R}^{2N}_+. \label{def:contraction}
\end{eqnarray}

Recall that $\hat{y}^{t_m *}$ denotes an optimizer of $\hL_{\phi}^{t_m}$. Since $\hL_{\phi}^{t_m}$ is a time-varying problem, consider  capturing the  variation of an optimizer over two consecutive time instants as:
\begin{eqnarray}
	\|\hat{y}^{t_{m+1} *}- \hat{y}^{t_m *}\|\leq \sigma, \label{eq:sigma}
\end{eqnarray}
where $0 < \sigma < + \infty$~\cite{SimonettoGlobalsip2014}.

Then, the following result  characterizes the discrepancy between the powers produced by (\ref{eq:online}) and an optimizer of $\hL_{\phi}^{t_m}$.

\begin{theorem}\label{the:online}
	Under Assumptions \ref{as:Cinv}--\ref{ass:error} and step sizes chosen according to Theorem \ref{the:convergence}, the sequence $\{{y}^{t_m}\}$ generated by Algorithm \ref{alg:mechanism_3} converges as
	\vspace{-.2cm}
		\begin{eqnarray}
	\lim_{m\rightarrow\infty}\sup\|{y}^{t_m}(K)-\hat{y}^{t_m *}\|=\frac{\|\rho\|}{1-\Delta}+\frac{\sigma\Delta^K}{1-\Delta^K}~~ .\label{eq:onlinerad}
	\end{eqnarray}
  \hfill $\Box$
\end{theorem}

\begin{proof}
We can characterize the distance between the operating point achieved by (\ref{eq:online}) in $K$ iterations and the optimizer of $\hL^{t_m}_{\phi}$ as follows:
	\begin{eqnarray}
	&&\|{y}^{t_m}(K)-\hat{y}^{t_m*}\|\nonumber\\
	&=&\|{T}^{t_m}({y}^{t_m}(K-1))-\hat{T}^{t_m}({y}^{t_m}(K-1))\nonumber\\
	&&~~+\hat{T}^{t_m}({y}^{t_m}(K-1))-\hat{y}^{t_m *}\|\label{eq:steptm3}\\
	&\leq&\|{T}^{t_m}({y}^t(K-1))-\hat{T}^{t_m}({y}^{t_m}(K-1))\|\nonumber\\
	&&~~+\|\hat{T}^{t_m}({y}^{t_m}(K-1))-\hat{y}^{t_m *}\|\label{eq:steptm3_1}\\
	&\leq& \|\rho\|+\Delta\|{y}^{t_m}(K-1)-\hat{y}^{t_m *}\|\label{eq:steptm3_2}\\
	&\leq& \frac{\|\rho\|(1-\Delta^K)}{1-\Delta}+\Delta^K\|{y}^{t_m}(0)-\hat{y}^{t_m *}\|\label{eq:steptm3_5}\\
	&=&\frac{\|\rho\|(1-\Delta^K)}{1-\Delta} +\Delta^K\|{y}^{t_{m-1}}(K)-\hat{y}^{t_{m-1}*}\nonumber \\
	&&~~+\hat{y}^{t_{m-1}*}-\hat{y}^{t_m*}\| \\
	&\leq&\frac{\|\rho\|(1-\Delta^K)}{1-\Delta}+\Delta^K\|{y}^{t_{m-1}}(K)-\hat{y}^{t_{m-1} *}\|\nonumber\\
	&&~~+\Delta^K\|{y}^{t_{m-1} *}-\hat{y}^{t_m *}\| \\
	&\leq& \frac{\|\rho\|(1-\Delta^K)}{1-\Delta}+\Delta^K\|{y}^{t_{m-1}}(K)-{y}^{t_{m-1}*}\| \nonumber \\
	&&~~ +\Delta^K\sigma,\label{eq:steptm4}
	\end{eqnarray}
	where: (\ref{eq:steptm3_2}) follows from (\ref{eq:bound2model}) and (\ref{def:contraction}); (\ref{eq:steptm3_5}) can be obtained by  repeating (\ref{eq:steptm3})--(\ref{eq:steptm3_2}) for $K$ times; and, (\ref{eq:steptm4}) follows from (\ref{eq:sigma}). {We repeat steps (\ref{eq:steptm3})--(\ref{eq:steptm4}) recursively over time instants $t_m,\ldots, t_0$ to obtain:}
	\begin{align}
	& \|{y}^{t_m}(K)-\hat{y}^{t_m *}\|\nonumber\\
	&\leq (\frac{\|\rho\|(1-\Delta^K)}{1-\Delta}+\Delta^K\sigma)\frac{1-\Delta^{K m}}{1-\Delta^K}+\Delta^{K m}\|{y}^{t_0}-\hat{y}^{t_0*}\|.\nonumber
	\end{align}
	When $m \rightarrow\infty$, (\ref{eq:onlinerad}) follows.
\end{proof}

The result~\eqref{eq:onlinerad} bounds the maximum discrepancy between the setpoints ${y}^{t_m}(K)$ generated by Algorithm~\ref{alg:mechanism_3} and a  time-varying optimizer of $(\bm{\mathcal{P}^{t}_2})$.  The bound~\eqref{eq:onlinerad}  depends on: 

\noindent i)  The underlying dynamics of the distribution system; in fact, we recall that the dynamics of non-controllable power assets, constraints, and operational conditions translate into temporal variations of the optimizers of $(\bm{\mathcal{P}^{t}_2})$~\cite{SimonettoGlobalsip2014}, which is characterized by  the parameter $\sigma$. When the variation of $\hat{y}^{t_m *}$ is smooth in time, bound~\eqref{eq:onlinerad} becomes tighter. 

\noindent ii) The approximation error introduced by the linearized power-flow equation, which is implicitly captured by $\rho$.

The result~\eqref{eq:onlinerad} can also be interpreted as an input-to-state stability, when one adopts the trajectory $\{\hat{y}^{t_m *}\}$ as a reference frame. Future research efforts will aim at characterizing the discrepancy between ${y}^{t_m}(K)$ and the optimal point of $(\bm{\mathcal{P}^{t}_2})$ when the nonlinear AC power-flow equations are utilized. 

Finally, when $K=1$,~\eqref{eq:onlinerad} boils down to:
\begin{align}
	\lim_{m\rightarrow\infty}\sup\|{y}^{t_m}(1)-\hat{y}^{t_m *}\|=\frac{\|\rho\| + \sigma\Delta}{1-\Delta} \, .\label{eq:onlinerad1}
\end{align}
The bound in~\eqref{eq:onlinerad1} is close in spirit to~\cite{OPFpursuit}, and it can be further simplified as $\frac{\|\rho\|}{1-\Delta}$ for the case of time-invariant (i.e. static) settings~\cite{SimonettoGlobalsip2014}.

\section{Numeric Examples}
\label{sec:scenarios}
\subsection{Simulation Setup}

Consider a modified version of the IEEE 37-node test feeder shown in Figure~\ref{F_feeder}. The modified network is obtained by considering the phase ``c'' of the original system  and by replacing the loads specified in the original dataset with real load data measured from feeders in Anatolia, California, during a week of August 2012~\cite{Bank13}. Particularly, the data have a granularity of 1 second, and represent the loading of secondary transformers.
Line impedances, shunt admittances, as well as active and reactive loads are adopted from the respective data set. It is assumed that 18 PV systems are located at nodes $4$, $7$, $10$, $13$, $17$, $20$, $22$, $23$, $26$, $28$, $29$, $30$, $31$, $32$, $33$, $34$, $35$, and $36$, and their generation profiles are simulated based on the real solar irradiance data available in~\cite{Bank13}. 
{The ratings of these inverters are $300$ kVA for $i = 3$, $350$ kVA for $i = 15, 16$, and $200$ kVA for the remaining inverters. Loads and the power available from a PV system with capacity of $50$~kW are reported in Fig.~\ref{F_VPP_paper_load} for illustrative purposes. }

\begin{figure}[t]
\centering
\vspace{.25cm}
\includegraphics[width=.45\textwidth]{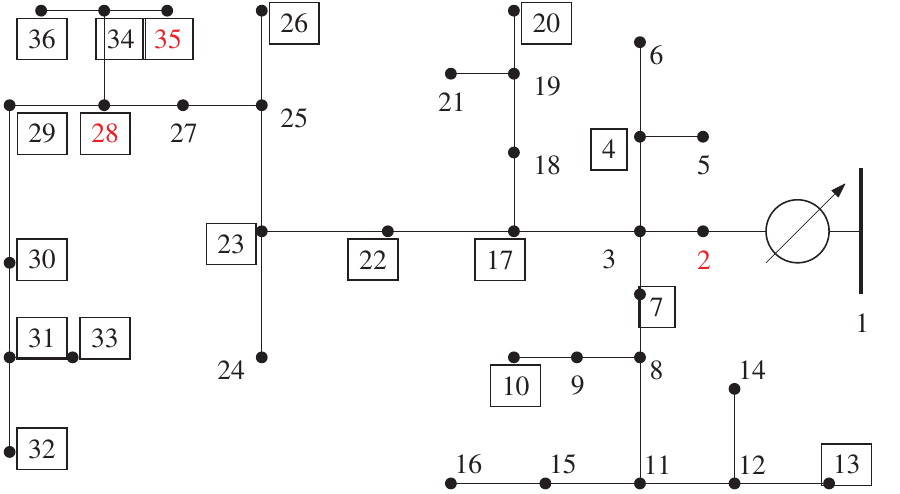}
\caption{IEEE 37-node feeder. The boxes represent PV systems. The red nodes are the ones analyzed in the numerical example.}
\label{F_feeder}
\end{figure}

\begin{figure}[t!] 
\begin{center}
\hspace{-.0cm}\includegraphics[width=9.6cm]{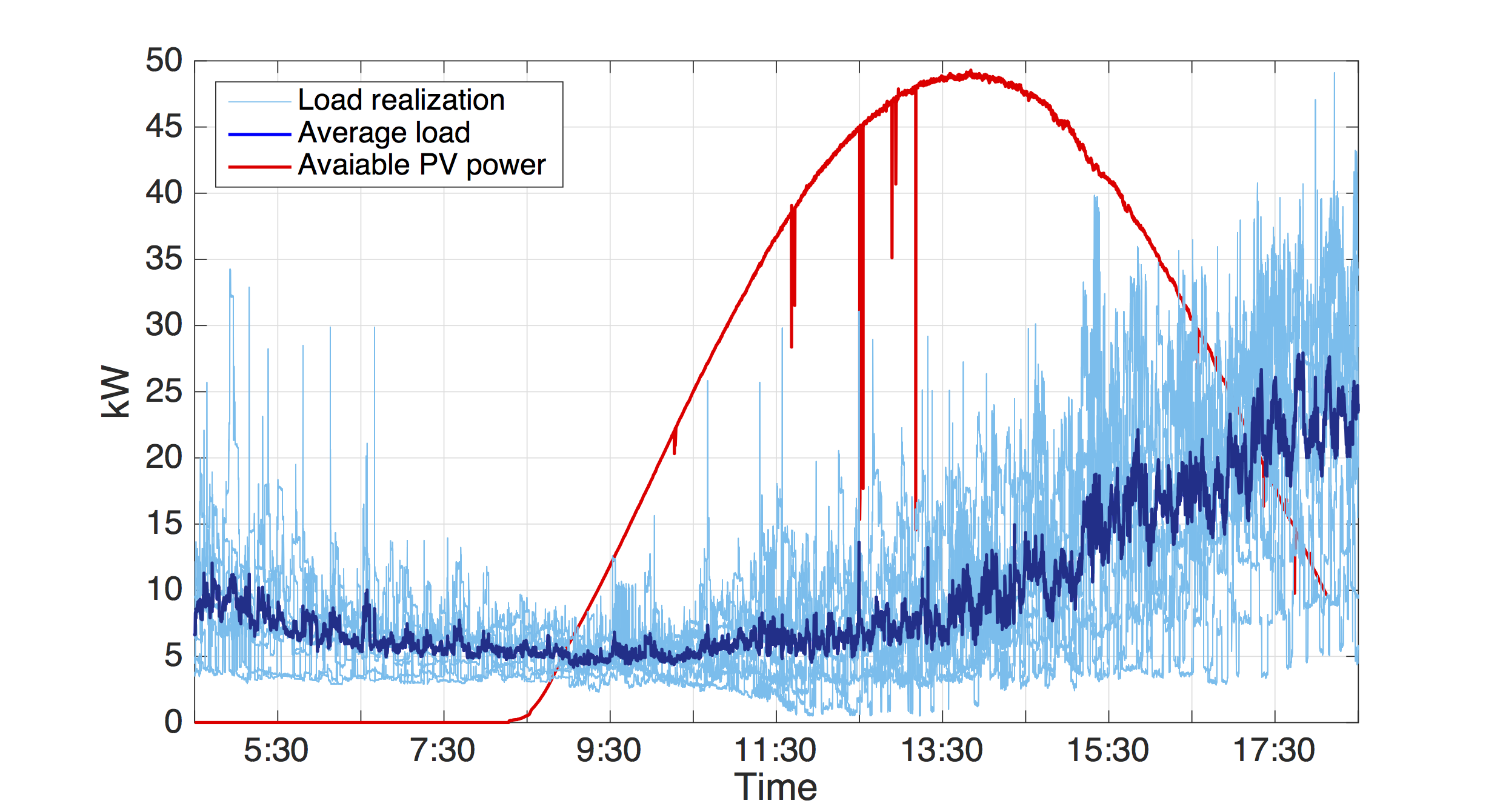}
\end{center}
\vspace{-.4cm}
\caption{Profiles of loads and power available from the PV systems. The average load profile is marked in blue.}
\vspace{-.2cm}
\label{F_VPP_paper_load}
\end{figure}

The voltage limits $\overline{v}_i$ and $\underline{v}_i$ are set to $1.05$ p.u. and $0.95$ p.u. respectively, for $\forall i\in\hN$.
Various step sizes $\varepsilon_1$ and $\varepsilon_2$ are tested to provide examples of cases where the algorithm converges as well as cases where it is not convergent. The customers' objective functions are set uniformly to $C_i^t(p_i^t,q_i^t) =   c_p (p^t_{i,\textrm{av}} - p^t_i)^2+c_q q_i^{t2} $, in an effort to minimize the amount of real power curtailed from the available power $p^t_{i,\textrm{pv}}$ based on irradiance conditions at time $t$, and the amount of reactive power injected or absorbed. The coefficients are set to $c_p = 3$ and $c_q = 1$. 
The network-oriented objective is set to $D(v^t)=\frac{1}{2}\|v^t-v^{\text{nom}}\|_2^2$ to penalize voltage deviation from the nominal value $v^{\text{nom}} = 1 $ p.u.
Without loss of generality, we demonstrate our results with the trade-off parameter $\gamma$ set to either $0$ or $1$. For $\gamma = 1$, it is possible to trade off  the customer-oriented objectives for flatness of the voltage profile. The regularization parameter $\phi$ is set to $10^{-4}$.
\begin{figure}
\centering
\includegraphics[trim = 0mm 0mm 0mm 0mm, clip, scale=0.44]{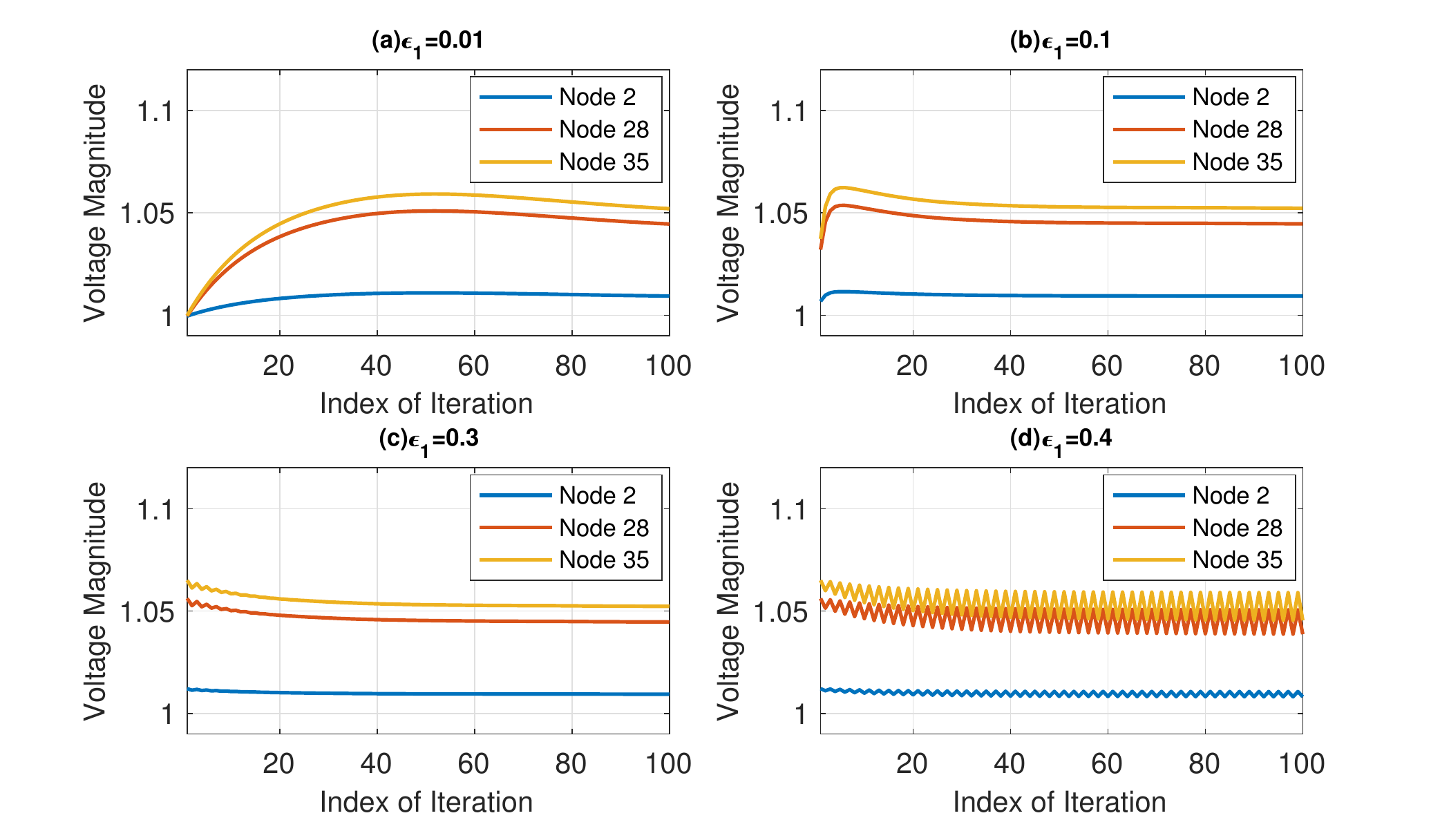} 
\caption{Convergence of the distributed algorithm with increasing step size $\varepsilon_1$ and fixed step size $\varepsilon_2$.}
\label{F_convergence}
\vspace{-.2cm}
\end{figure}

\subsection{Iterative Algorithm}
We first test Algorithm~1 and show how the algorithm can address overvoltages in distribution systems~\cite{Liu08}. To this end, we focus on a single timeslot at 12~pm. 

\subsubsection{Convergence}\label{sec:simconv}
Let  $\gamma=0$ for simplicity. Recall from Theorem~\ref{the:convergence} that step sizes $\varepsilon_1$ and $\varepsilon_2$ both affect the convergence properties. For simplicity, set $\varepsilon_2=0.01$, and consider tuning $\varepsilon_1$ to achieve convergence. Similar results can be observed by fixing $\varepsilon_1$ and tuning $\varepsilon_2$, or tuning both $\varepsilon_1$ and $\varepsilon_2$. As shown in Figure~\ref{F_convergence}, when $\varepsilon_1$ is increased from $0.01$ to $0.3$, we observe faster convergence. However, when we further increase $\varepsilon_1$ beyond $0.4$, an oscillatory behavior is observed.

\begin{figure}
\centering
\includegraphics[trim = 0mm 0mm 0mm 0mm, clip, scale=0.45]{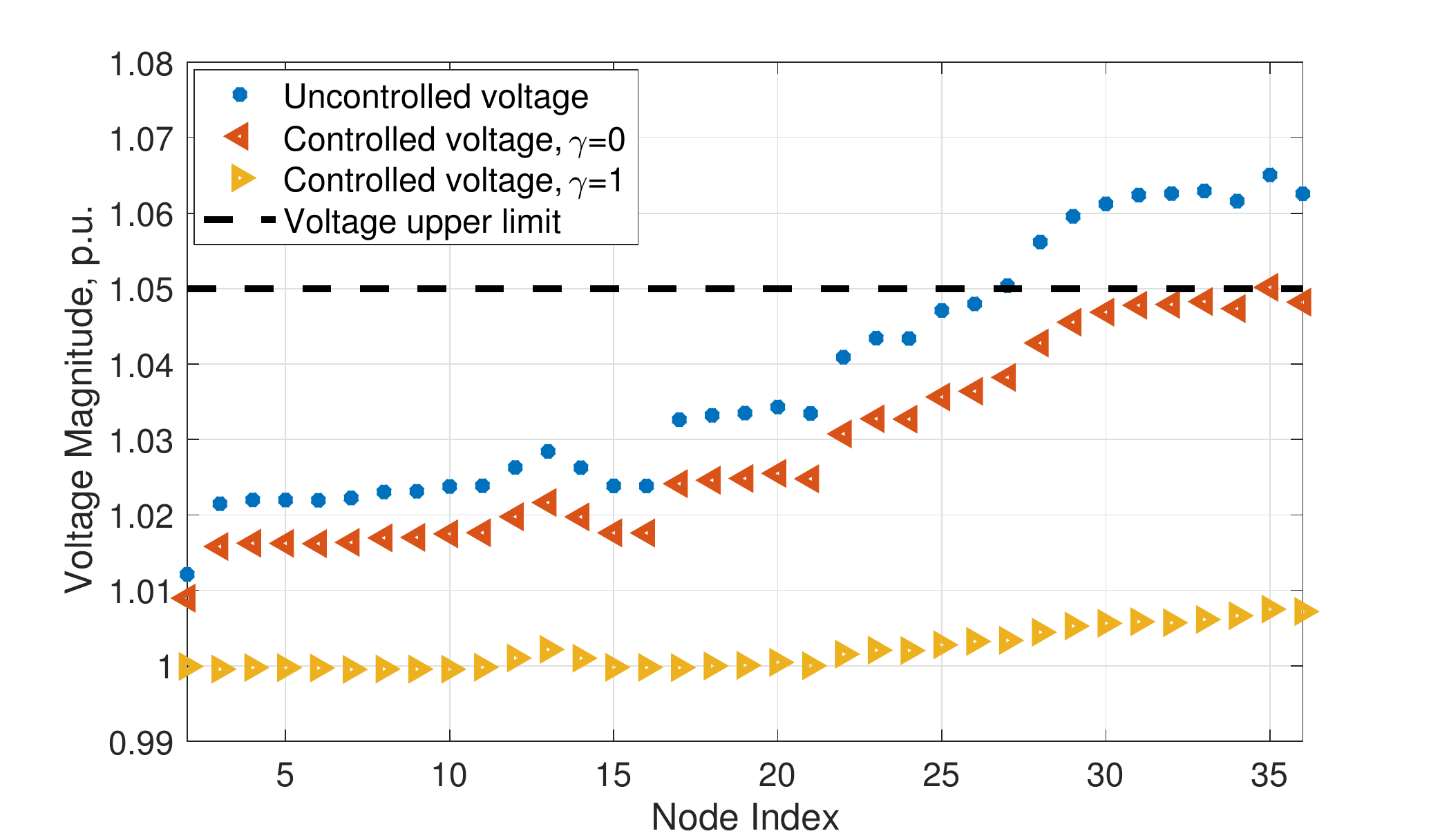} 
\caption{Controlled and uncontrolled voltages at all buses at noon.}
\label{F_convergence2}
\vspace{-.2cm}
\end{figure}

\begin{figure}
	\centering
	\includegraphics[trim = 0mm 0mm 0mm 0mm, clip, scale=0.33]{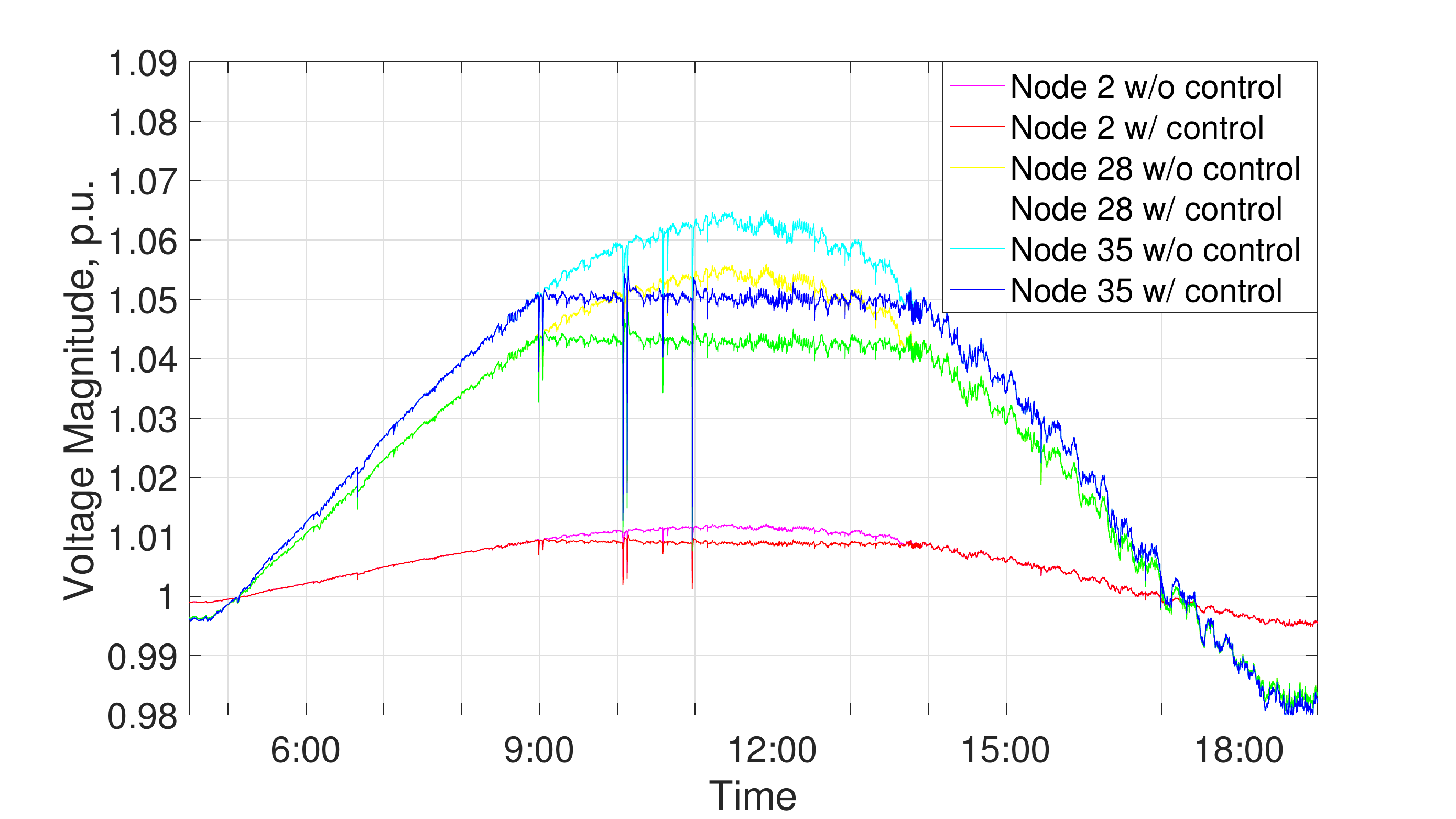} 
	\caption{Controlled and uncontrolled voltages at nodes 2, 28, and 35 from 4:30~am to 7:00~pm with $\gamma=0$ and $K=1$.}
	\label{time_gamma0}
	\includegraphics[trim = 0mm 0mm 0mm 0mm, clip, scale=0.33]{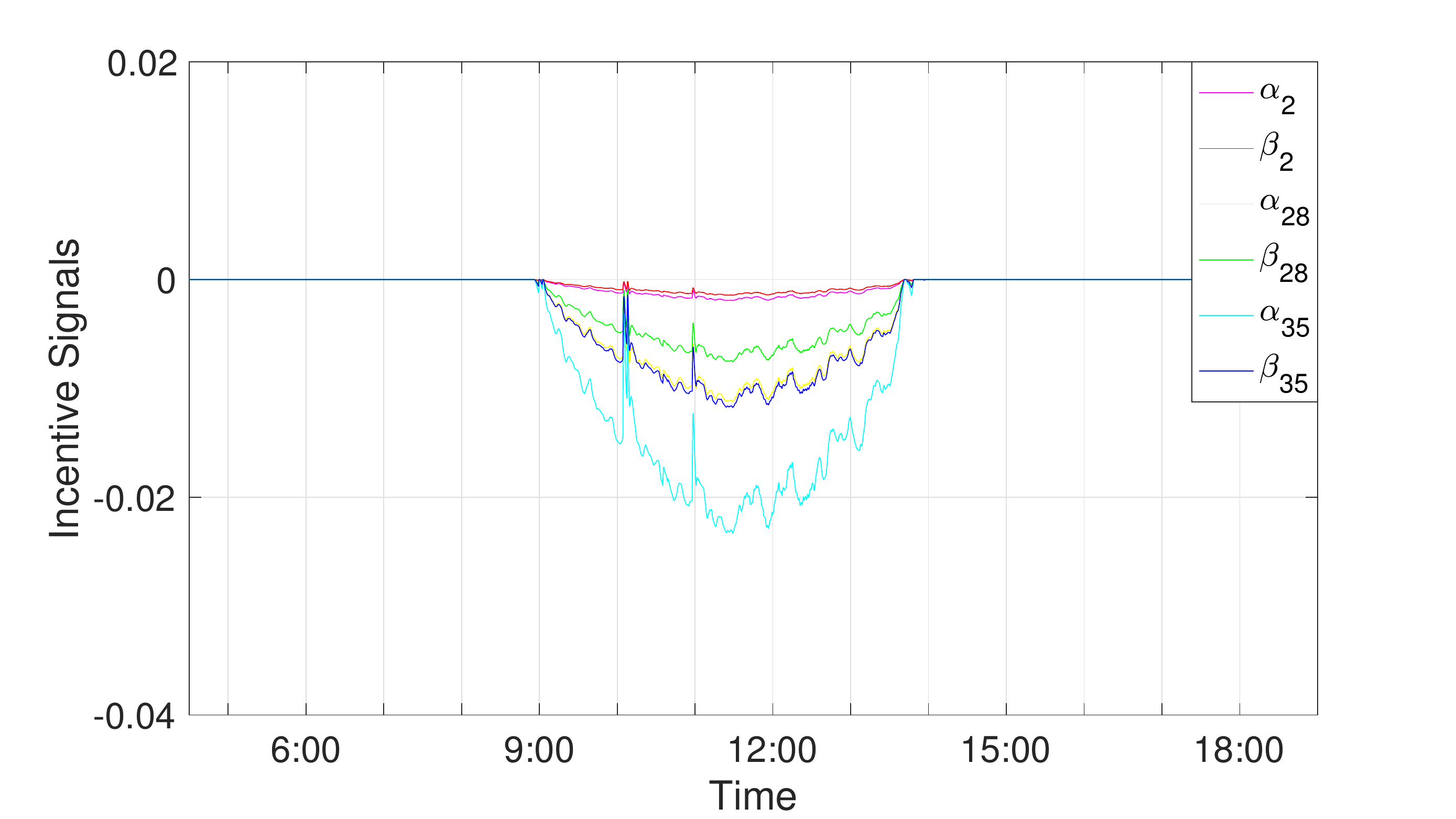} 
	\caption{The incentive signals at nodes 2, 28, and 35 from 4:30~am to 7:00~pm with $\gamma=0$ and $K=1$.}
	\label{F_signal0}
	\includegraphics[trim = 0mm 0mm 0mm 0mm, clip, scale=0.33]{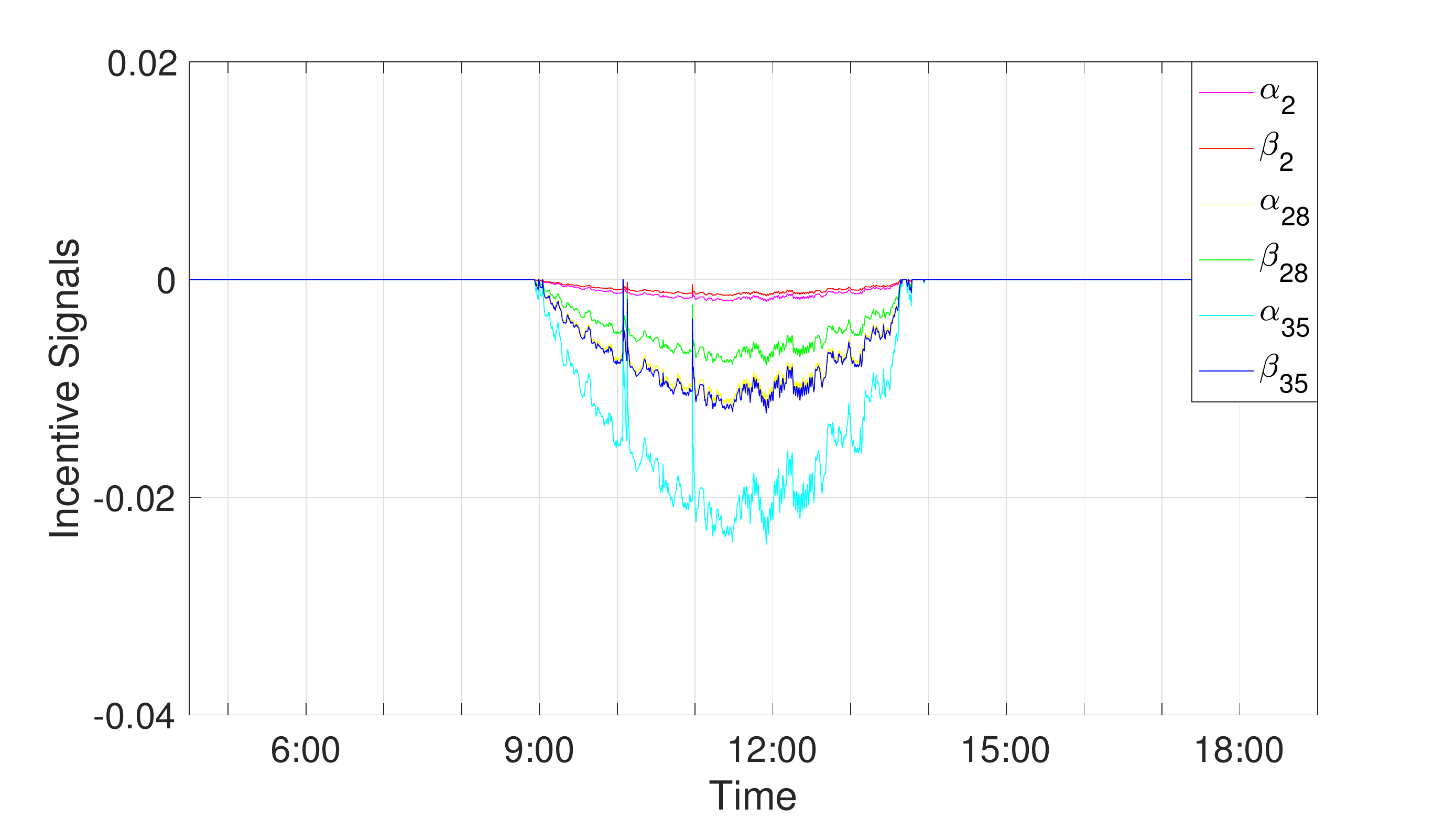} 
	\caption{The incentive signals at nodes 2, 28, and 35 from 4:30~am to 7:00~pm with $\gamma=0$ and  $K=5$.}                                               	\label{F_signal0_K5}
\end{figure}

\subsubsection{Voltage regulation}
The results are plotted in Figure~\ref{F_convergence2} corresponding to the case where $\varepsilon_1=\varepsilon_2=0.01$. We show voltage profiles in three scenarios: (i) uncontrolled setting, where the PV systems operate at unity power factor and inject the maximum available power without any curtailment (blue dots), (ii) controlled voltages with $\gamma=0$ (red dots), and (iii) controlled voltages with $\gamma=1$ (yellow dots). 
It is clear that in the uncontrolled case (i) the voltage values exceed the limit of 1.05~p.u.  (black dashed line) due to large reverse power flows, while the controlled scenarios (ii) and (iii) show voltage within limits.
Furthermore, voltage values achieved by (iii) are closer to the nominal value than those by (ii), because (iii) also penalizes voltage deviation from 1~p.u.

\begin{figure}
	\centering
	\includegraphics[trim = 0mm 0mm 0mm 0mm, clip, scale=0.33]{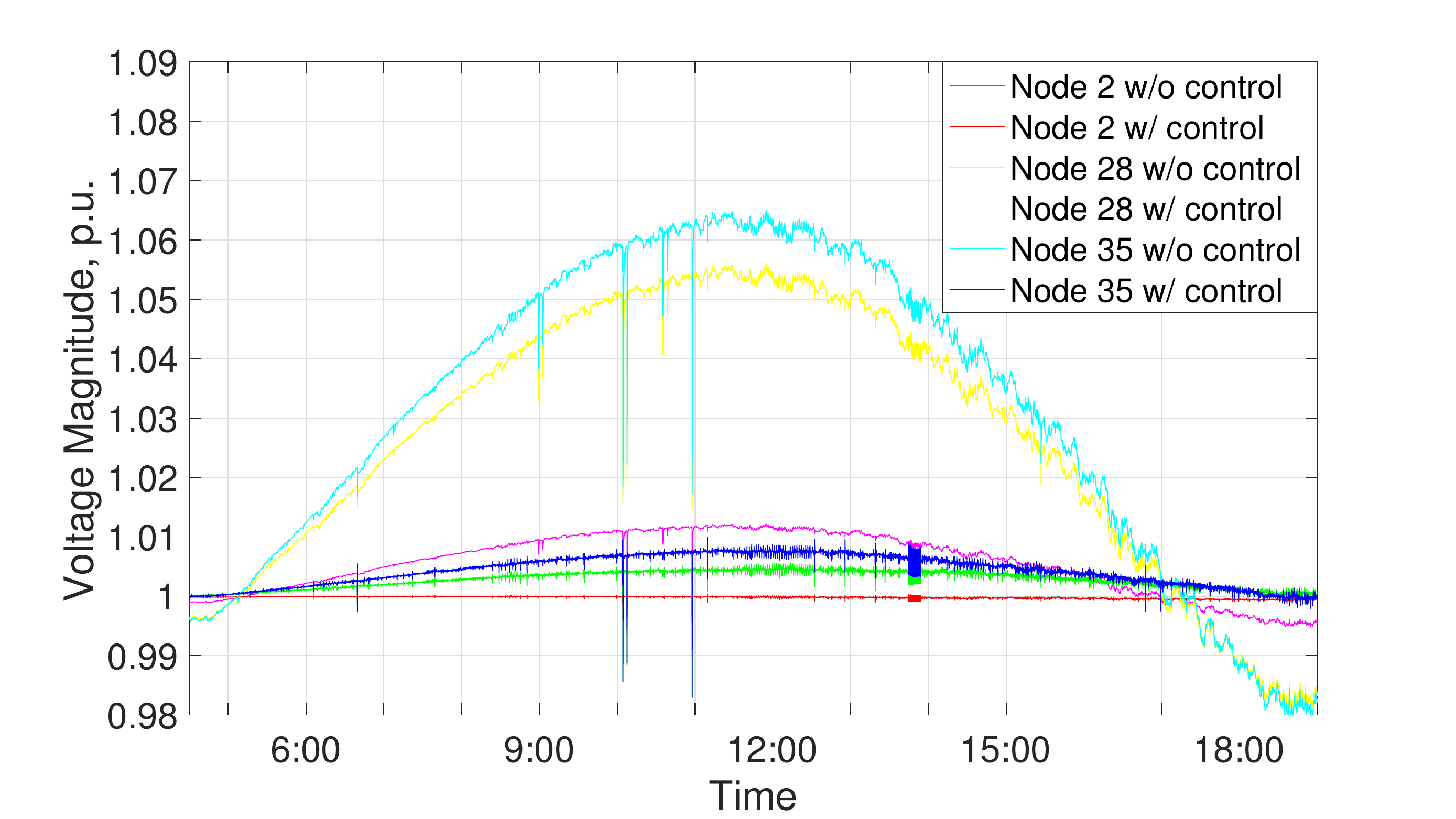} 
	\caption{Controlled and uncontrolled voltages at nodes 2, 28, and 35 from 4:30 am to 7:00 pm with $\gamma=1$ and $K=1$.}
	\label{time_gamma1}
	\includegraphics[trim = 0mm 0mm 0mm 0mm, clip, scale=0.33]{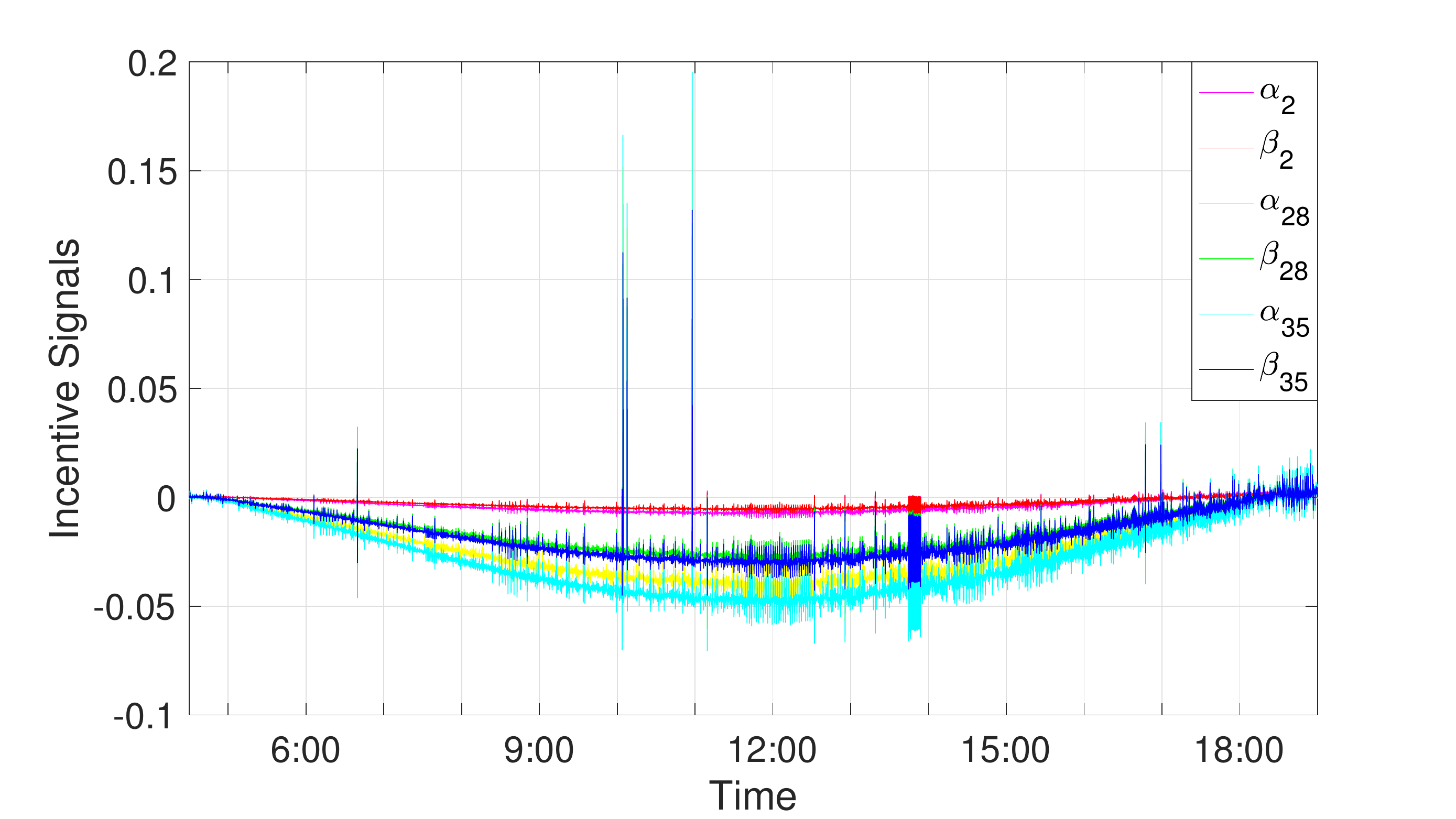} 
	\caption{The incentive signals at nodes 2, 28, and 35 from 4:30~am to 7:00~pm with $\gamma=1$ and $K=1$.}
	\label{F_signal1}
	\includegraphics[trim = 0mm 0mm 0mm 0mm, clip, scale=0.33]{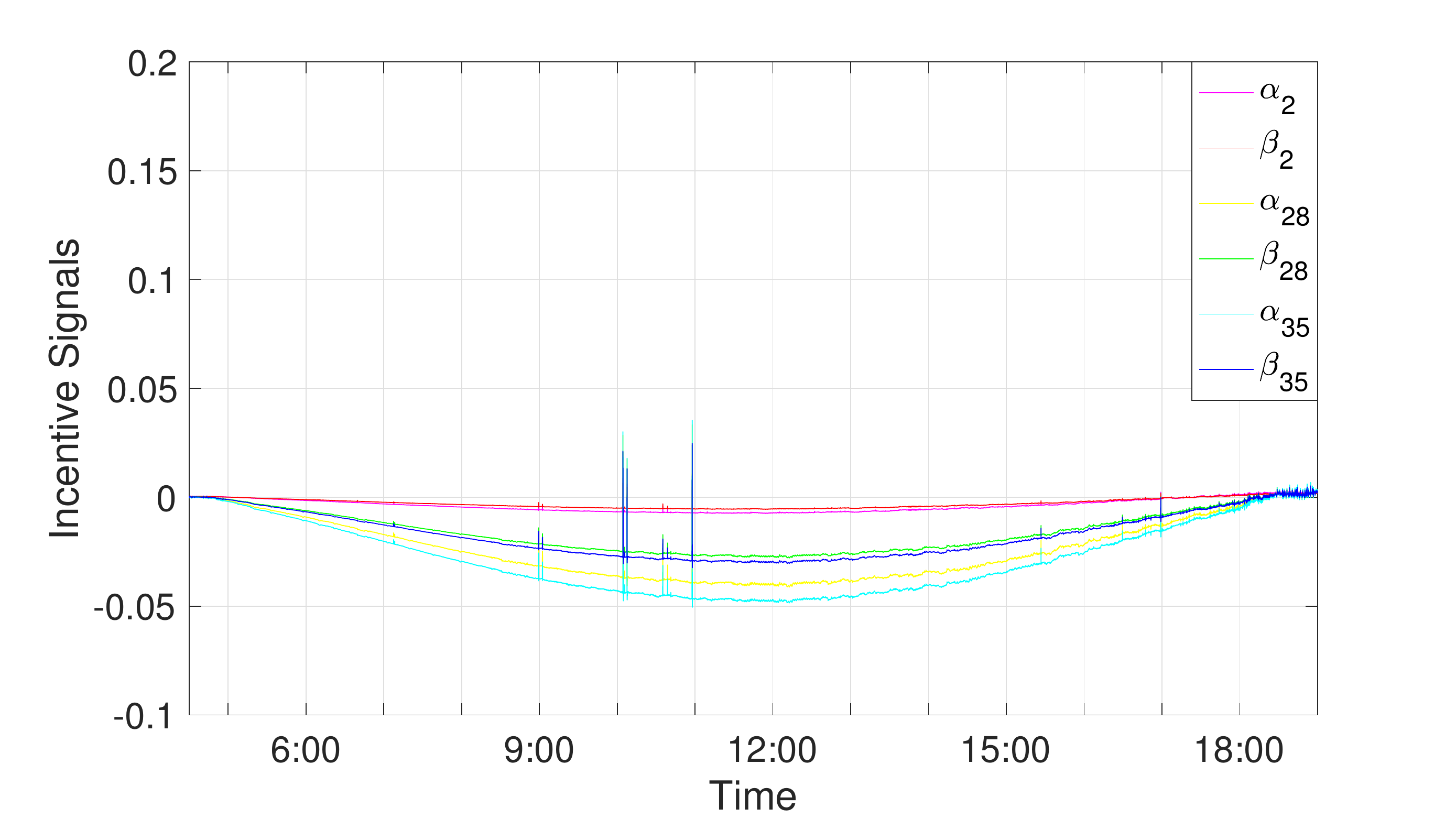} 
	\caption{The incentive signals at nodes 2, 28, and 35 from 4:30~am to 7:00~pm with $\gamma=1$ and $K=5$.}
	\label{F_signal1_K5}
\vspace{-.2cm}
\end{figure}

\subsection{Online Algorithm}\label{sec:numvary}

Next, Algorithm~2 is tested based on the irradiance and load profiles shown in Figure~\ref{F_VPP_paper_load}. One iteration (i.e., $K=1$) is performed every second (i.e., $h = 1$ second). In the following,  the performance of the proposed online algorithm are illustrated for both cases of $\gamma = 0$ and $\gamma = 1$. We will provide the voltage profiles as well as the profiles of the incentive signals. In addition, incentive signal profiles under $K=5$ are provided for comparative purpose.

\subsubsection{$\gamma=0$}
In this case, the function $D(v)$ is disregarded. The voltage profiles obtained when the PV inverters operate according to business-as-usual practices and when they implement the proposed Algorithm~2  are provided for nodes 2, 28, and 35 in Fig.~\ref{time_gamma0}. In the uncontrolled case, voltage values exceed the upper limit during the mid-day hours because of the reverse power flows; in contrast, the proposed algorithm enforces voltage regulation, even when only one iteration is performed every second. Fig.~\ref{F_signal0} illustrates that the incentive  signals become nonzero when voltages would violate the limits. The negative signals incentivize the customers to curtail active power and produce negative reactive power. 

\subsubsection{$\gamma=1$}
The voltage profiles obtained when $\gamma = 1$ are plotted in Fig. \ref{time_gamma1}. In this case, the voltage magnitudes are driven closer to the nominal value, at the cost of curtailing more real power and absorbing more reactive power.  Voltages are clearly within limits. 

\subsubsection{$K=5$}
We repeat the simulations with five iterations per second (i.e., $K=5$), and plot the signal profiles based on the last iteration of each second. The results are presented in Fig.~\ref{F_signal0_K5} and Fig.~\ref{F_signal1_K5}. As expected, the incentive signals generated with more iterations provide more accurate (see Fig.~\ref{F_signal0} vs Fig.~\ref{F_signal0_K5}) and more steady (see Fig.~\ref{F_signal1} vs Fig.~\ref{F_signal1_K5}) tracking of voltage changes.
The resultant controlled voltage profiles with $K=5$ are omitted because they are not largely different from Fig.~\ref{time_gamma0} and Fig.~\ref{time_gamma1} with $K=1$; nevertheless, by examining the results in details we have found that voltage profiles with $K=5$ commit less voltage violation when $\gamma=0$, and enjoy smaller (time) variance in general. 

\section{Conclusion}
\label{sec:conclusions}
This paper considers a time-varying social welfare maximization problem modeling network operator and DER-owners operational  objectives as well as voltage constraints. The formulated problem is non-convex; however, we  propose a convex relaxation and we provide conditions under which the optimal solutions of the relaxed problem coincide with the optimal points of the non-convex social-welfare problem. We then design distributed algorithms to identify the solutions of the time-varying social welfare maximization problem. An online  algorithm is {proposed} to enable tracking of the solutions in the presence of fast time-varying operational conditions and changing optimization objectives. Stability of the proposed schemes is analytically established and numerically corroborated. {Future research directions include the extension of the proposed framework to control DERs with discrete power levels and devices involving discrete decision variables.}

{
\appendix 

Complementing the convergence results of Theorem~\ref{the:convergence}, in the following we provide sufficient conditions on the stepsizes $\varepsilon_1$ and $\varepsilon_2$ that guarantee the operator $\hat{T}$ in (\ref{eq:mappingT}) to be a contraction.

\begin{theorem}\label{the:stepsizebound}
If the stepsizes $\varepsilon_1$ and $\varepsilon_2$ satisfy the following conditions for any $i\in\hN$:
\begin{subequations} \label{eq:stepsizecon1}
\begin{eqnarray}
&&\hspace{-12mm}\varepsilon_2<\frac{1}{2\sum_{j\in\hN}R_{ij}},\  \varepsilon_1\!\nabla^2_{p_i}(C_i+\gamma D)>2\varepsilon_2\sum_{j\in\hN}R_{ij},\\[-3pt]
&&\hspace{-12mm}\varepsilon_1\nabla^2_{p_i}(C_i+\gamma D)+2\varepsilon_2\!\sum_{j\in\hN}R_{ij}<2,\\[-6pt]
&&\hspace{-12mm}\varepsilon_2<\frac{1}{2\sum_{j\in\hN}X_{ij}},\ \varepsilon_1\!\nabla^2_{q_i}(C_i+\gamma D)>2\varepsilon_2\sum_{j\in\hN}X_{ij},\\[-3pt]
&&\hspace{-12mm}\varepsilon_1\nabla^2_{q_i}(C_k+\gamma D)+2\varepsilon_2\!\sum_{j\in\hN}X_{ij}<2,\\[-6pt]
&&\hspace{-12mm}\varepsilon_1<\frac{1}{\sum_{j\in\hN}(R_{ij}+X_{ij})},\ \varepsilon_1\!\sum_{j\in\hN}(R_{ij}+X_{ij})>\varepsilon_2\phi,\\[-2pt]
&&\hspace{-12mm}\varepsilon_1\sum_{j\in\hN}(R_{ij}+X_{ij})+\varepsilon_2\phi<2,
\end{eqnarray}
\end{subequations}
then $\hat{T}$ is a contraction.\hfill$\Box$
\end{theorem}

\begin{proof}
Let $\nabla\hat{T}\in\mathbb{R}^{4N\times 4N}$ denote Jacobian matrix of $\hat{T}$, and let $\nabla\hat{T}_{ij}$ denote the  element on row $i$ and column $j$ of matrix $\nabla\hat{T}$. To prove that $\hat{T}$ is a contraction, it is sufficient to have the following condition:
\begin{eqnarray}
\sum_j |\nabla\hat{T}_{ij}|<1,\ \forall i,\nonumber
\end{eqnarray}
which is satisfied if the following three inequalities hold:
\begin{subequations}\label{eq:stepsizecon2}
\begin{eqnarray}
\big|1-\varepsilon_1(\nabla^2_{p_i}(C_i+\gamma D))\big|+2\varepsilon_2\sum_{j\in\hN}R_{ij}<1,\\
\big|1-\varepsilon_1(\nabla^2_{q_i}(C_i+\gamma D))\big|+2\varepsilon_2\sum_{j\in\hN}X_{ij}<1,\\
\big|1-\varepsilon_2\phi\big|+\varepsilon_1\sum_{j\in\hN}(R_{ij}+X_{ij})<1.
\end{eqnarray}
\end{subequations}
Conditions (\ref{eq:stepsizecon1}) and (\ref{eq:stepsizecon2}) are in fact equivalent. Therefore, (\ref{eq:stepsizecon1}) are sufficient for $\hat{T}$ to be a contraction.
\end{proof}
\begin{remark}
Conditions (\ref{eq:stepsizecon1}) together with assumptions in this paper guarantee the existence of small enough step sizes $\varepsilon_1$ and $\varepsilon_2$ to achieve convergence. This result is consistent with Theorem~\ref{the:convergence}.
\end{remark}
}

\IEEEtriggeratref{32} 
\bibliographystyle{IEEEtran}
\bibliography{biblio_pricing.bib,biblio_2.bib}

\end{document}

%% file: intro_v13.tex


Market-based algorithms have been recently developed to control distributed energy assets with the objective of incentivizing end-customers to provide  services to the grid while maximizing economic benefits and performance objectives~\cite{MohsenianRad10,Maharjan13, Pedram14}. For example, end-customers may be incentivized to adjust the output powers of distributed energy resources (DERs) in real time to aid voltage regulation~\cite{Liu08}, control the aggregate network demand~\cite{li2016market}, and follow regulating signals~\cite{Vrettos16}.  

This paper aims to design incentive-based distributed algorithms that allow network operator and end-customers to pursue given operational and economic objectives and, in doing so, ensure that voltage magnitudes are within the prescribed limits. We start with the formulation of a \emph{time-varying} social welfare maximization problem that captures a variety of optimization objectives, hardware constraints, and the nonlinear power-flow equations governing the physics of distribution systems. The time-varying nature of the problem~\cite{SimonettoGlobalsip2014, OPFpursuit} enables us to model optimization and operational objectives  that vary in time and to capture variability of ambient conditions and non-controllable energy assets; the time-varying problem thus defines optimal  \emph{trajectories} for the active and reactive powers of the DERs as well as voltage levels. A linear approximation of the nonlinear power-flow equations~\cite{Baran89, Christakou13, swaroop2015linear, bolognani2015linear} is utilized to facilitate the development of a computationally-tractable algorithms. Even when linear power-flow models are adopted, the resultant problem is non-convex; {however, we  propose a convex relaxation and provide conditions under which the relaxation is exact, i.e., the optimal solutions of the relaxed problem coincide with the global optimal points of the non-convex social-welfare problem.} We then design distributed algorithms to identify (and track) the solutions of the time-varying social welfare maximization problem. {The  algorithm enables a  
 distributed solution of the social-welfare-maximization problem where: (i) customers do not share private information such as their cost function and the feasible  set of the DERs' output powers; and (ii) customers and network operator pursue their own economic and operational objectives, while ensuring that voltage limits are systematically satisfied throughout the network. } 

The first algorithm is applicable to problems that vary slowly in time, where \emph{offline} iterative methods can be utilized to solve sampled instances~\cite{SimonettoGlobalsip2014} of the time-varying problem to convergence.  An \emph{online}  algorithm is then {proposed} to enable tracking of the solutions in the presence of fast time-varying operational conditions and changing optimization objectives. The online algorithm involves a strategy where the network operator collects voltage measurements {through} the feeder to build incentive signals for the DER-owners in real time;  DERs then adjust  the generated/consumed powers in order to avoid the violation of the voltage constraints while maximizing given objectives. Stability of the proposed schemes is analytically established and numerically corroborated. The design of the algorithms is grounded on the  decomposability of  the primal-dual gradient algorithms, and convergence of the distributed algorithm to the optimal solution of the social-welfare maximization problem is shown by leveraging the contraction mapping arguments with a regularized Lagrangian. 

It should be pointed out that, traditionally, voltage regulation problems arising from reverse power flows~\cite{Liu08} have been tackled by considering local Volt/VAR and Volt/Watt controllers~{\cite{farivar2013equilibrium, farivar2015local, zhou2015pseudo, zhou2016incremental, Zhu15, Zhang13, baker2017network}}  or optimization-based techniques~\cite{swaroop2015linear,Farivar12,OID} that leverage the flexibility of power-electronics-interfaced renewable sources of energy in adjusting the output real and reactive powers. Compared to local control strategies, the proposed method allows the network operator and end-customers to pursue {well-defined} performance objectives; compared to existing optimization strategies, the proposed method casts the voltage regulation problem within a pricing/incentive realm, and provides insights as to how to design real-time pricing/incentive schemes.  

Load control  {is} formulated as a Stackelberg game in e.g.,~\cite{chen2012optimal,tushar2015three}, and~\cite{li2016market}, and conditions for convergence of a leader-follower strategy to an equilibrium point {are} derived. Based on a Stackelberg game formulation, a distributed algorithm with only local information available for both utility companies and DER-owners {is} developed in~\cite{Maharjan13}. A two-level game setting {is} considered in~\cite{chai2014demand}. A time-varying pricing strategy is proposed in~\cite{li2011optimal} to maximize social welfare.  However,  the strategies outlined in~\cite{Maharjan13,chen2012optimal,tushar2015three,li2016market,chai2014demand,li2011optimal} and pertinent references therein are network agnostic, in the sense that AC power flows in the power network are ignored (and DERs are assumed to be connected to one single electrical node). It follows that network-agnostic methods do not account for voltage variations induced  by the controlled  DERs and must be complemented by voltage-regulation mechanisms. The frameworks proposed in, e.g.,~\cite{li2011optimal,LinaCDC15} offer a way to account for the power flows, but their applicability is limited to a restricted class of network topologies.

As for online optimization methods for distribution systems,  Centralized controllers  are developed in~\cite{Bernstein15,AndreayOnlineOpt}, based on continuous gradient steering algorithms; the framework accounts for errors in the implementable power setpoints, and convergence of the average setpoints to the minimum of the considered control objective is established. An online gradient algorithm for AC optimal power flow in single-phase radial networks is proposed in~\cite{LowOnlineOPF}; it is shown that the proposed algorithm converges to the set of local optima of a static AC OPF problem, and sufficient conditions under which the online OPF converges to a global optimum are provided.
A real-time control strategy that enables DERs to maximize given performance objectives is proposed in~\cite{OPFpursuit}. The proposed online algorithm is close in spirit to~\cite{OPFpursuit}; {however compared to~\cite{OPFpursuit}, this paper casts the real-time voltage regulation problem within a time-varying game-theoretic framework and {develops} distributed strategies based on pricing/incentive signals;~\cite{OPFpursuit} does not address the design of pricing/incentive signals}. 

The paper is organized as follows. Section~\ref{sec:model} introduces the system model and presents the problem formulation. Section~\ref{sec:static} focuses on the design of iterative algorithms that afford an offline implementation, while Section~\ref{sec:realtime} presents the online algorithm. Section~\ref{sec:scenarios} outlines results from numerical experiments and Section~\ref{sec:conclusions} concludes the paper. Preliminary results were presented in~\cite{zhou2017pricing}.